\documentclass[12pt]{amsart}
\usepackage{amsmath,amssymb}

\usepackage{calrsfs}
\usepackage{url}
\newtheorem{thm}{Theorem}[section]

\newtheorem{lem}[thm]{Lemma}
\newtheorem{prop}[thm]{Proposition}

\theoremstyle{definition}

\newtheorem{exmp}[thm]{Example}

\newtheorem{rem}[thm]{Remark}
\numberwithin{equation}{section}

\newcommand{\Z}{{\mathbb{Z}}}
\newcommand{\Q}{{\mathbb{Q}}}
\newcommand{\C}{{\mathbb{C}}}
\newcommand{\F}{{\mathbb{F}}}
\newcommand{\fg}{{\mathfrak{g}}}
\newcommand{\fh}{{\mathfrak{h}}}

\newcommand{\be}{{\mathbf{e}}}

\newcommand{\fS}{{\mathfrak{S}}}

\newcommand{\cE}{{\mathcal{E}}}

\newcommand{\cN}{{\mathcal{N}}}

\DeclareMathOperator{\Irr}{Irr}
\renewcommand{\leq}{\leqslant}
\renewcommand{\geq}{\geqslant}

\begin{document}
%

\title[Green functions in small characteristic]{Computing
Green functions in small characteristic}
\author[Meinolf Geck]{Meinolf Geck}
\address{IAZ - Lehrstuhl f\"ur Algebra, Universit\"at Stuttgart, 
Pfaffenwaldring 57, D--70569 Stuttgart, Germany}
\email{meinolf.geck@mathematik.uni-stuttgart.de}

\date{April 10, 2019}

\dedicatory{To the memory of Kay Magaard}

\begin{abstract} 
Let $G(q)$ be a finite group of Lie type over a 
field with $q$ elements, where $q$ is a prime power. The Green functions of 
$G(q)$, as defined by Deligne and Lusztig, are known in \textit{almost} 
all cases by work of Beynon--Spaltenstein, Lusztig und Shoji. Open cases 
exist for groups of exceptional type ${^2\!E}_6$, $E_7$, $E_8$ in small 
characteristics. We propose a general method for dealing with these cases, 
which procedes by a reduction to the case where $q$ is a prime and then 
uses computer algebra techniques. In this way, all open cases in type 
${^2\!E}_6$, $E_7$ are solved, as well as at least one particular open
case in type $E_8$. 
\end{abstract}
\maketitle

\section{Introduction} \label{sec0}

Let $G(q)$ be a group of Lie type over a finite field with $q$ elements. 
We are concerned with the problem of computing the Green functions of 
$G(q)$, as defined by Deligne and Lusztig \cite{DeLu}. This is an important
and essential part of the more general problem of computing the whole 
character table of $G(q)$; see Carter's book \cite{C2} for further background.
There is a long tradition of work on Green functions; the principal 
ideas and methods, which remain valid and state of the art as of today, are 
summarized in Shoji's survey \cite{S1} from 1987. At that point, the Green 
functions where known in all cases where $q$ is a power of a good prime for 
$G$, or where $q$ is arbitrary, $G$ is of small rank and the whole table
of unipotent character values is available (like for $G_2$, ${^3\!D}_4$,
${^2\!B}_2$, ${^2\!G}_2$). Subsequently, explicit results for $F_4$, 
${^2\!F}_4$, $E_6$ in small characteristics were obtained by Malle 
\cite{Ma90}, \cite{Mal1} and Porsch \cite{Por}.

Regarding the general theory, it was first shown by Lusztig \cite{L5} 
(with some mild restrictions on~$q$) and then by Shoji \cite{S2},
\cite{S3} (in complete generality) that the original Green functions of 
\cite{DeLu} can be identified with another type of Green functions defined 
in terms of Lusztig's character sheaves \cite{Lintr}. This provides new,
extremely powerful tools. In this setting, groups of classical type in 
characteristic~$2$ are dealt with by Shoji \cite{S6}. Thus, the remaining 
open cases are as follows:
\begin{equation*}
{^2\!E}_6(3^m), \quad E_7(2^m),\quad E_7(3^m),
\quad E_8(2^m), \quad E_8(3^m),\quad E_8(5^m)\tag{$\heartsuit$}
\end{equation*}
for any $m\geq 1$. (See Marcelo--Shinoda \cite{MaSh} for some comments
about the Green functions of $F_4(3^m)$.) In principle, one could try to 
deal with these cases by similar methods as in the papers by Malle and 
Porsch mentioned above; however, these involve the technically complicated 
and unpleasant task of explicitly inducing class functions from proper 
subgroups. In this paper, we use another approach, similar to that in 
\cite{pbad}. By \cite[Theorem~3.7]{hartur}, the computation of the Green 
functions of $G(q)$, where $q=p^m$ with $m\geq 1$, can be reduced to the 
base case where $m=1$ which amounts to just six individual cases which can 
be addressed by computer algebra methods. In this way, we will solve all the
open cases for the groups of type ${^2\!E}_6$, $E_7$ in the list 
($\heartsuit$), as well as one particular case for type $E_8$ in 
characteristic~$2$.

In Section~\ref{sec1}, we review the general plan for computing 
Green functions, which reduces matters to the determination  of 
certain ``$Y$-functions''. In Section~\ref{sec2}, we discuss a number of 
techniques for determining these functions. Consequently, we obtain a method
for solving the remaining open problems which relies on knowing at 
least some values of the permutation character of $G(q)$ on the cosets 
of a Borel subgroup $B(q)\subseteq G(q)$. In order to compute such values,
we shall work with an explicit realisation of $G(q)$ as a matrix group. In
Section~\ref{sec3}, we advertise a ``canonical'' way of constructing $G(q)$,
following \cite{mylie}, \cite{L19}. Then the remaining sections deal with 
the discussion of the various cases in groups of type $F_4$, $E_6$, 
${^2\!E}_6$, $E_7$; see Section~\ref{sec7} for the particular case in 
type $E_8$.

We heavily rely on Michel's version of {\sf CHEVIE} \cite{jmich}, as well 
as programs (written by the author in {\sf GAP} \cite{gap}) implementing 
the constructions in Section~\ref{sec3}. As far as the remaining open cases
in type $E_8$ are concerned, it seems that the above method might work
in principle, but more sophisticated algorithms will certainly be required.
(For example, one could replace the Borel subgroup $B(q)$ by a parabolic 
subgroup of $G(q)$.) This will be discussed elsewhere. 

The main computational challenge of our approach is the explicit
computation of the values of the above-mentioned permutation character 
of $G(q)$. For this purpose, we need to count the (left) cosets of $B(q)$
that are fixed by a given element $g\in G(q)$. But, because of the sheer 
size of the groups in question (e.g., for $G(q)=E_7(3)$ we have $[G(q):B(q)]
\approx 17\cdot 10^{30}$), it is entirely impossible to run through the 
complete list of cosets. Now a crucial feature of our approach is that, 
typically, we only need to obtain \textit{lower bounds} 
for the number of fixed points, and this can be exploited as follows. 
By the sharp form of the Bruhat decomposition, we have a partition 
\begin{center}
$G(q)=\coprod_{w\in W} B(q)wB(q)$,
\end{center} 
where $W$ is the Weyl group of $G(q)$ and each double coset $B(q)wB(q)$ 
contains precisely $q^{l(w)}$ left $B(q)$-cosets; here, $l(w)$ is the length 
of~$w$. Now we simply begin with various elements $w\in W$ of relatively
small length, run through the left cosets that are contained in $B(q)wB(q)$, 
and check if they are fixed by~$g$ or not. In a sense, we were just lucky
because in all cases that we consider, this is sufficient to reach the 
desired lower bounds for the total number of fixed points~---~and there
are cases where we never reached the exact total number, even after weeks 
or months of computations. (We will indicate the maximum length required 
for Weyl group elements in all cases in Sections~\ref{sec4}--\ref{sec7}.)

\medskip
\textit{Acknowledgements}. The author is indebted to Gunter Malle for
a careful reading of the manuscript and for a number of useful comments. 
This work is a contribution to the SFB-TRR 195 ``Symbolic Tools in 
Mathematics and their Application'' of the German Research Foundation (DFG). 

\section{On the computation of Green functions} \label{sec1}

Let $p$ be a prime and $k=\overline{\F}_p$ be an algebraic closure of the
field with $p$ elements. Let $G$ be a connected reductive algebraic group 
over $k$ and assume that $G$ is defined over the finite subfield $\F_q
\subseteq k$, where $q=p^m$ for some $m\geq 1$. Let $F\colon G\rightarrow 
G$ be the corresponding Frobenius map. Let $B_0 \subseteq G$ be an 
$F$-stable Borel subgroup and $T_0\subseteq B_0$ be an $F$-stable maximal
torus. Let $W=N_G(T_0)/T_0$ be the corresponding Weyl group. 
For each $w\in W$, let $R_w$ be the virtual representation of the finite 
group $G^F$ defined by Deligne--Lusztig \cite[\S 1]{DeLu}. (In the setting 
of \cite[\S 7.2]{C2}, we have $\mbox{Tr}(g,R_w)=R_{T_w,1}(g)$ for $g\in G^F$,
where $T_w\subseteq G$ is an $F$-stable maximal torus obtained from $T_0$ 
by twisting with~$w$, and $1$ stands for the trivial character of $T^F$.) 
This construction is carried out over $\overline{\Q}_\ell$, an algebraic
closure of the $\ell$-adic numbers where $\ell$ is a prime not equal to~$p$.
The corresponding Green function is defined by 
\[ Q_w\colon G_{\text{uni}}^F \rightarrow \overline{\Q}_\ell,
\qquad u \mapsto \mbox{Tr}(u,R_w),\]
where $G_{\text{uni}}$ denotes the set of unipotent elements of $G$.
It is known that $Q_w(u)\in \Z$ for all $u\in G_{\text{uni}}^F$;
see \cite[\S 7.6]{C2}. So the character formula \cite[7.2.8]{C2} shows
that we also have $\mbox{Tr}(g,R_w)\in \Z$ for all $g\in G^F$.
The general plan for computing the values of $Q_w$ is explained in
\cite[Chap.~24]{L2e}, \cite[\S 5]{S1}, \cite[1.1--1.3]{S6a} 
(even for generalised Green functions, which we will not consider here).
In order to be able to address the main open issues, we will have to go 
through some of the steps of that plan, where we streamline the exposition
as much as possible and concentrate on the algorithmic aspects.

\begin{rem} \label{sub21}
The Frobenius map $F$ induces an automorphism of $W$ which we denote by
$\gamma\colon W\rightarrow W$. Let $\Irr(W)$ be the set of irreducible
representations of $W$ over $\overline{\Q}_\ell$ (up to isomorphism). Let
$\Irr(W)^\gamma$ be the set of all those $E\in \Irr(W)$ that are
$\gamma$-invariant, that is, there exists a bijective linear map $\sigma_E
\colon E\rightarrow E$ such that $\sigma_E {\circ} w=\gamma(w){\circ} \sigma_E
\colon E\rightarrow E$ for all $w\in W$. Note that $\sigma_E$ is only
unique up to scalar multiples but, if $\gamma$ has order $d\geq 1$, then
one can always find $\sigma_E$ such that
\[\sigma_E^d=\mbox{id}_E\qquad\mbox{and}\qquad \mbox{Tr}(\sigma_E{\circ} w,
E) \in \Z\quad\mbox{for all $w\in W$};\]
see \cite[3.2]{L1}. In what follows, we assume that a fixed choice of
$\sigma_E$ satisfying the above conditions has been made for each $E\in
\Irr(W)^\gamma$. (For example, one could take the ``preferred'' choice for
$\sigma_E$ specified by Lusztig \cite[17.2]{L2d}.) 
\end{rem}

For $E\in \Irr(W)^\gamma$, the corresponding almost character is the class 
function $R_{E}\colon G^F\rightarrow \overline{\Q}_\ell$ defined by
\[ R_{E}(g):=\frac{1}{|W|} \sum_{w\in W} \mbox{Tr}(\sigma_E {\circ} w,
E)\mbox{Tr}(g,R_w) \qquad \mbox{for all $g\in G^F$}.\]
Since all the terms $\mbox{Tr}(\sigma_E {\circ} w,E)$ are integers, all the
values $R_{E}$ are in $\Q$. By \cite[3.9]{L1}, the above functions are 
orthonormal with respect to the standard inner product on class functions 
of $G^F$. Furthermore, by \cite[3.19]{cbms}, we have
\[ Q_w(u)=\sum_{E\in \Irr(W)^\gamma} \mbox{Tr}(\sigma_E 
{\circ} w,E) R_{E}(u)\qquad \mbox{for $w\in W$, $u\in 
G_{\text{uni}}^F$}.\]
Hence, knowing the values of all Green functions $Q_w$ is equivalent to
knowing the values of all $R_{E}$ on $G_{\text{uni}}^F$. We define the matrix
$\tilde{\Omega}=(\tilde{\omega}_{E',E})_{E',E \in \Irr(W)^\gamma}$ where
\[\tilde{\omega}_{E',E}:=\frac{1}{|W|} \sum_{w\in W} [G^F:T_w^F]\, 
\mbox{Tr}(\sigma_{E'}{\circ} w, E') \mbox{Tr}(\sigma_E{\circ} w,E)\in\Q;\]
here, $T_w\subseteq G$ denotes an $F$-stable maximal torus obtained from
$T_0$ by twisting with~$w$ and the maps $\sigma_E\colon E\rightarrow E$,
$\sigma_{E'}\colon E'\rightarrow E'$ are as above.

\begin{prop}[Orthogonality relations] \label{prop21} For $E,E'\in
\Irr(W)^\gamma$, we have
\[ \tilde{\omega}_{E',E}=\sum_{g\in G_{\operatorname{uni}}^F} R_{E'}(g) 
R_E(g).\]
\end{prop}

\begin{proof} Arguing as in \cite[3.19]{cbms}, the above relations are a 
formal consequence of the orthogonality relations for the Green 
functions $Q_w$ in \cite[Prop.~7.6.2]{C2}.
\end{proof}

Let $\cN_G$ be the set of all pairs $(C,\cE)$ where $C$ is a unipotent
class in $G$ and $\cE$ is a $G$-equivariant irreducible
$\overline{\Q}_\ell$-local system on $C$ (up to isomorphism). The Springer 
correspondence defines an injective map
\[ \iota_G\colon \Irr(W)\hookrightarrow \cN_G; \]
see Lusztig \cite{LuIC}, \cite[Chap~24]{L2e}, and the references there.
Let $E\in \Irr(W)$ and $\iota_G(E)=(C,\cE)\in\cN_G$. Then we define
\[ d_E:=(\dim G-\dim C-\dim T_0)/2.\]
Note that $\dim C_G(g)\geq \dim T_0$ for $g\in G$. Furthermore,
$d_E\in\Z_{\geq 0}$ since $\dim G-\dim T_0$ is always even and so is
$\dim C$; see  \cite[\S 5.10]{C2} and the references there. 
Now assume that $E\in \Irr(W)^\gamma$. Then $F(C)=C$ and
$F^*\cE\cong \cE$. We define a function 
\[Y_E\colon G_{\text{uni}}^F \rightarrow \Q\]
as follows. Let $g\in G_{\text{uni}}^F$. Then we set $Y_E(g):=0$ if 
$g\not\in C$, and 
\[ Y_E(g):=q^{-d_E}R_E(g) \qquad\mbox{if $g\in C^F$}.\]
Now we can state the following fundamental result.

\begin{thm}[Lusztig, Shoji] \label{thmls} In the above setting, the 
following hold.
\begin{itemize}
\item[(a)] The functions $\{Y_E\mid E \in \Irr(W^\gamma)\}$ are 
integer-valued and linearly independent. 
\item[(b)] There are unique coefficients $p_{E',E}\in\Z$ ($E',E\in
\Irr(W)^\gamma)$ such that
\[ R_{E}|_{G_{\operatorname{uni}}^F}=\sum_{E'\in \Irr(W)^\gamma}  q^{d_E}\,
p_{E',E} Y_{E'}\qquad\mbox{for all $E\in \Irr(W)^\gamma$}.\]
\item[(c)] We have $p_{E,E}=1$; furthermore, $p_{E',E}=0$ if 
$E'\neq E$ and $d_{E'}\geq d_E$.
\end{itemize}
\end{thm}

\begin{proof} By Lusztig \cite{L5} (with some mild restrictions on~$p,q$) 
and Shoji \cite{S2}, \cite{S3} (in complete generality), the original Green 
functions of \cite{DeLu} can be identified with another type of Green 
functions defined in terms of character sheaves \cite{Lintr}. 
So we can place ourselves in the setting of \cite[Chap.~24]{L2e}. Thus, the 
restrictions of the functions $R_E$ to $G_{\text{uni}}^F$ are indeed the 
characteristic functions of the character sheaves $A_i$ in 
\cite[24.2]{L2e}. Furthermore, the functions $Y_E$ defined above are
indeed equal to the functions $Y_i$ in \cite[24.2.3]{L2e}; thus, 
if $\iota_G(E)=(C,\cE)$, then we have 
\[ Y_E(g)=\mbox{Tr}(\psi_g,\cE_g)\qquad \mbox{for $g\in C^F$},\]
where $\cE_g$ is the stalk of $\cE$ at~$g$ and $\psi_g\colon 
\cE_g\rightarrow\cE_g$ is a certain linear map of finite order. In
particular, this shows that the values of $Y_E$ are algebraic integers. 
Then all of the above statements follow from 
\cite[24.2, 24.3, 24.5]{L2e} and \cite[Theorem~24.4]{L2e}. Note 
that the hypotheses of \cite[Theorem~24.4]{L2e} (``cleanness'') are always 
satisfied by the main result of \cite{L10}. (Since we are only dealing 
with Green functions of $G^F$, and not with generalised Green functions, 
it would actually be sufficient to refer to \cite[\S 3]{aver} instead
 of \cite{L10}; see also \cite[\S 2]{hartur} where all of the above are 
discussed in somewhat more detail.)
\end{proof}

\begin{rem} \label{sub26} Lusztig \cite[\S 24.4]{L2e} describes a purely 
combinatorial algorithm for computing the coefficients $p_{E',E}$, which 
modifies and simplifies an earlier algorithm of Shoji \cite[\S 5]{S1}.
For this purpose, we define three matrices
\[ P=(p_{E',E}),\qquad  \Omega= (\omega_{E',E}), \qquad 
\Lambda=(\lambda_{E',E}), \]
where, in each case, the indices run over all $E',E\in\Irr(W)^\gamma$. Here,
$p_{E',E}$ are the coefficients in Theorem~\ref{thmls}; furthermore,
\[\omega_{E',E}:=q^{-d_E-d_{E'}} \tilde{\omega}_{E',E}\qquad\mbox{and}
\qquad \lambda_{E',E}:=\sum_{g\in G_{\text{uni}}^F} Y_{E'}(g)Y_E(g).\]
Then the orthogonality relations in Proposition~\ref{prop21} give rise
to the matrix identity 
\[P^{\operatorname{tr}}\cdot \Lambda\cdot P=\Omega; \qquad\mbox{see 
Lusztig \cite[24.9]{L2e}, Shoji \cite[5.6]{S1}}.\]
In general, given the right hand side $\Omega$, such a system of equations 
will not have a unique solution for $P,\Lambda$. But if we take into account 
the additional information on the coefficients $p_{E',E}$ in 
Theorem~\ref{thmls}(c), then it does have a unique solution, which can be 
found by a recursive algorithm. 
\end{rem}

\begin{rem} \label{jean} The Springer correspondence is explicitly known 
in all cases; see the tables in Carter \cite[\S 13.3]{C2}, Lusztig 
\cite{LuIC}, Lusztig--Spaltenstein \cite{LuSp1}, Spaltenstein \cite{Spa1} 
(and the further references there). It can be obtained electronically, via 
tables or combinatoral algorithms, through Michel's version of the
{\sf CHEVIE} system \cite{jmich}; see the function {\tt UnipotentClasses}.
There is also an implementation of the algorithm in Remark~\ref{sub26};
see the function {\tt ICCTable}. Examples will be given below.
\end{rem}

\begin{exmp} \label{sub26a} Let $E_1\in \Irr(W)$ be the trivial 
representation of $W$. Clearly, we have $E_1\in \Irr(W)^\gamma$;
furthermore, we can certainly choose $\sigma_{E_1}\colon E_1\rightarrow E_1$ 
to be the identity map.  Then, with this choice, we have
\[ R_{E_1}(g)=1 \qquad\mbox{for all $g\in G^F$};\]
see \cite[7.14.1]{DeLu} or \cite[Prop~7.4.2]{C2}. It is also known
that $\iota_G(E_1)=(C_{\text{reg}},\overline{\Q}_\ell)$ where $C_{\text{reg}}$
is the class of regular unipotent elements (see, e.g., \cite[1.1]{Spa1});
hence, we have 
\[ d_{E_1}=0\qquad\mbox{and}\qquad R_{E_1}|_{G_{\text{uni}}^F}=
\sum_{E'\in \Irr(W)^\gamma} p_{E',E}Y_{E'}.\]
Combining the above two expressions for $R_{E_1}$, the 
functions $Y_{E'}$ are determined for all $E'\in \Irr(W)^\gamma$ such that 
$p_{E',E_1}\neq 0$. Indeed, if $p_{E',E_1}\neq 0$, then let us write 
$\iota_G(E')=(C',\cE')$ where $C'$ is an $F$-stable unipotent class. Since 
the functions $\{Y_E \mid E\in \Irr(W^\gamma)\}$ are linearly independent, 
we conclude that 
\[ p_{E',E_1}Y_{E'}(g)=1 \qquad\mbox{for all $g\in C'^F$}.\]
Since $p_{E',E}\in\Z$ and $Y_{E'}(g)\in \Z$ for all $g\in G^F$, we either 
have $Y_{E'}(g)=1$ for all $g\in C'^F$, or $Y_{E'}(g)=-1$ for all $g\in C'^F$,
where the sign is determined by $p_{E',E_1}$.
\end{exmp}

\begin{table}[htbp] \caption{The Springer correspondence and $p_{E',E}$ 
for $G_2$, $p=3$} \label{tabG23} 
\begin{center} $\begin{array}{cccc}\hline E & d_E& A(u) &\iota_G(E) \\ \hline 
E_{1,6} & 6 & \{1\}& (1,\overline{\Q}_\ell) \\ 
E_{1,3}' & 3 & \{1\}&  ((\tilde{A}_1)_3,\overline{\Q}_\ell) \\
E_{1,3}'' & 3& \{1\} &(A_1,\overline{\Q}_\ell) \\
E_{2,2} & 2 & \{1\}& (\tilde{A}_1,\overline{\Q}_\ell) \\
E_{2,1} & 1 & \Z/2\Z &(G_2(a_1),\overline{\Q}_\ell) \\
E_{1,0} & 0& \Z/3\Z &(G_2,\overline{\Q}_\ell) \\ \hline\end{array}\qquad
\begin{array}{c@{\hspace{5pt}}c@{\hspace{5pt}}c@{\hspace{5pt}}
c@{\hspace{5pt}}c@{\hspace{5pt}}c@{\hspace{5pt}}c}
\hline p_{E',E} & E_{1,6} & E_{1,3}' & E_{1,3}'' & 
E_{2,2} & E_{2,1} & E_{1,0}\\\hline
E_{1,6} & 1 & 1 & 1 & q^2{+}1 & q^4{+}1 & 1 \\ 
E_{1,3}' & 0 & 1 & 0 & 1 & 1 & 1 \\
E_{1,3}'' & 0 & 0 & 1 & 1 & 1 & 1 \\
E_{2,2} & 0 & 0 & 0 & 1 & 1 & 1\\
E_{2,1} & 0 & 0 & 0 & 0 & 1 & 1\\
E_{1,0} & 0 & 0 & 0 & 0 & 0 & 1 \\ \hline\end{array}$
\end{center}
\end{table}

\begin{exmp} \label{sub27} Let $G$ be of type $G_2$ and $p=3$. In this 
case, $W=\langle s_1, s_2\rangle$ where $s_1$ is the reflection 
corresponding to a long simple root and $s_1$ is the reflection 
corresponding to a short simple root. We have
\[ \Irr(W)=\{E_{1,0},E_{1,6},E_{1,3}',E_{1,3}'',E_{2,1},E_{2,2}\}\]
where $E_{1,0}$ is the trivial representation, $E_{1,6}$ is the sign
representations, $E_{1,3}'$, $E_{1,3}''$ are two further one-dimensional
representations such that 
\[ \mbox{Tr}(s_1,E_{1,3}')=\mbox{Tr}(s_2,E_{1,3}'')=-1 \quad\mbox{and}
\quad \mbox{Tr}(s_1, E_{1,3}'')=\mbox{Tr}(s_2, E_{1,3}')=1;\]
finally, $E_{2,1}$ is the standard reflection representation and $E_{2,2}$ is
a further two-dimensional representation. The Frobenius map $F$ acts 
trivially on $W$ and so $\gamma=\mbox{id}_W$. In Michel's version of
{\sf CHEVIE} \cite{jmich}, we obtain the information on unipotent classes 
and the Springer correspondence as follows.

\begin{verbatim}
    gap>  W := CoxeterGroup("G",2);; 
    gap>  uc := UnipotentClasses(W,3);;  # p=3
    gap>  Display(uc);  Display(ICCTable(uc));
\end{verbatim}

The information is summarized in Table~\ref{tabG23} (see also
\cite[p.~329]{Spa1}). There are $6$ unipotent classes, denoted by
$G_2, G_2(a_1)$, $\tilde{A}_1$, $A_1$, $(\tilde{A}_1)_3$, $1$.
Since $\gamma=\mbox{id}_W$, we also have  $\sigma_E=\mbox{id}_E$ for
all $E\in \Irr(W)$. Thus, we obtain explicit expressions
\[ R_E|_{G_{\text{uni}}^F}=\sum_{E'\in \Irr(W)} q^{d_E}\,p_{E',E}Y_{E'}
\qquad\mbox{for all $E\in \Irr(W)$}.\]
It remains to determine  the values of $Y_E$ for all $E\in \Irr(W)$. In
the present case, this is easily done using Example~\ref{sub26a}. 
Indeed, since all entries in the last column of Table~\ref{tabG23} are 
equal to~$1$, we have $R_{E_{1,0}}=\sum_{E\in \Irr(W)} Y_{E}$. Hence, 
each function $Y_{E}$ is identically~$1$ on $C^F$ where $\iota_G(E)=(C,
\overline{\Q}_\ell)$. 
\end{exmp}

In general, the determination of the functions $Y_E$ is a very subtle
problem. In order to solve it, one either needs further geometric 
information (as, for example, in Beynon--Spaltenstein \cite[\S 3, 
Case~V]{BeSp}, Shoji \cite[\S 1]{S6}) or some additional information about 
character values of $G^F$, which was readily availaible in the above 
example but may require much more work in other cases (as, for example,
in Malle \cite{Mal1}). The following discussion, which is inspired by the
approach of Marcelo--Shinoda \cite{MaSh}, will turn out to be very
useful in later sections.

\begin{rem} \label{rem31}
For $w=1$, the virtual representation $R_1$ of Deligne--Lusztig is known
to be an actual representation, which is in fact the permutation
representation of $G^F$ on the cosets of $B_0^F$ (see \cite[1.5]{DeLu}
or \cite[7.2.4]{C2}). Thus, for any unipotent element $u\in G^F$, we have
\begin{align*}
Q_1(u)&=\mbox{Tr}(u,R_1)=|\{gB_0^F\in G^F/B_0^F\mid ugB_0^F=gB_0^F\}|\\
&= |\{gB_0^F\in G^F/B_0^F\mid g^{-1}ug\in B_0^F\}|=
\sum_{1\leq i\leq r}\frac{|C_G(u)^F|}{|C_{B_0}(u_i)^F|}
\end{align*}
where $u_1,\ldots,u_r\in B_0^F$ are representatives of the conjugacy
classes of $B_0^F$ that are contained in the $G^F$-conjugacy class of~$u$.
On the other hand, expressing $R_1$ as a linear combination of $R_E$'s
and then using Theorem~\ref{thmls}(b), we obtain that 
\[ Q_1(u)=\sum_{E'\in \Irr(W)^\gamma} \tilde{p}_{E'} Y_{E'}(u)
\quad\mbox{where}\quad \tilde{p}_{E'}:=\sum_{E\in\Irr(W)^\gamma}
q^{d_E} \mbox{Tr}(\sigma_E)\, p_{E',E}.\]
Thus, since the terms $\tilde{p}_{E'}$ are determined by the algorithm 
in Remark~\ref{sub26}, this yields conditions on the values of
the functions $Y_{E'}$, once we manage to obtain some information
about $Q_1(u)$ by other means. 
\end{rem}

In Sections~\ref{sec4}--\ref{sec6}, we will try to evaluate $Q_1(u)$ 
explicitly for certain unipotent elements~$u$ (see Example~\ref{expg2} 
below for a first illustration). For this purpose, we note that the 
formula for $Q_1(u)$ can be further refined using the Bruhat decomposition. 
Let $\Phi$ be the root system of $G$ with respect to $T_0$. Let $\Phi^+ 
\subseteq \Phi$ be the set of positive roots determined by the choice of 
$B_0$. Finally, let $\{\alpha_i\mid i \in I\} \subseteq \Phi^+$ be the 
corresponding set of simple roots, where $I$ is a finite indexing set. 
This defines a length function $l\colon W\rightarrow \Z_{\geq 0}$. For 
each $\alpha\in \Phi$, let $U_\alpha=\{x_\alpha(t)\mid t\in k\} \subseteq G$
be the corresponding root subgroup. For $w\in W$, we set 
\[ U_w:=\langle U_\alpha \mid \alpha \in \Phi_w^+\rangle\subseteq G
\qquad\mbox{where}\qquad \Phi_w^+:=\{\alpha
\in\Phi^+\mid w(\alpha)\in \Phi^-\};\]
we also fix a representative $\dot{w}$ of $w$ in $N_G(T_0)$. Here, we 
tacitly assume that  $\dot{w}$ is chosen such that $F(\dot{w})=\dot{w}$ 
whenever $\gamma(w)=w$, which is possible by \cite[p.~33]{C2}. Then
we have the following sharp form of the Bruhat decomposition.
\[ G=\coprod_{w\in W} U_w\dot{w}B_0 \qquad\mbox{(with uniqueness of
expression});\]
see \cite[Theorem~2.5.14 and Prop.~2.5.6]{C2}. Writing 
$\Phi_w^+=\{\beta_1,\ldots,\beta_l\}$ where $l=l(w)$, we actually have 
$U_w=U_{\beta_1}\cdots U_{\beta_l}$ with uniqueness
of expression.

\begin{lem} \label{rem31aa} Let $u\in G^F$ be unipotent. Then 
\[Q_1(u)=\sum_{w\in W,\gamma(w)=w} |Q_{1,w}(u)|\]
where $Q_{1,w}(u):=\{ v \in U_w^F\mid \dot{w}^{-1}v^{-1}uv\dot{w} 
\in B_0^F\}$ for all $w \in W$ such that $\gamma(w)=w$.
\end{lem}

\begin{proof} By \cite[\S 2.9]{C2}, we also have a sharp form of the 
Bruhat decomposition for the finite group $G^F$, such that $G^F=
\coprod_{w} U_w^F\dot{w}B_0^F$, where the union runs over all $w\in W$ 
such that $\gamma(w)=w$. Inverting elements, we see that 
\[ \{v\dot{w}^{-1}\mid w \in W, \gamma(w)=w, v \in U_w^F \}\]
is a complete set of representatives of the cosets $\{gB_0^F\mid 
g\in G^F\}$. This yields the above formula.
\end{proof}

\begin{rem} \label{rem31f} As far as explicit computations using a computer
are concerned, the above formula means that
\[Q_1(u)\geq \sum_w |Q_{1,w}(u)|\]
where $w$ runs over all elements of $W$ of any given bounded length. Note 
that $|U_w^F|=q^{l(w)}$ (see \cite[p.~74]{C2}) which quickly becomes very 
large with growing $l(w)$. Thus, we can only reasonably work with bounds 
like $l(w)\leq 25$ (if $q=2$) or $l(w)\leq 16$ (if $q=3$) on a standard
computer. In any case, the above estimate will be crucial in our discussion 
of groups of exceptional type in Sections~\ref{sec4}--\ref{sec7}. 
\end{rem}

\section{On the determination of the functions $Y_E$} \label{sec2}

We will assume from now on that $G$ is simple and that the Frobenius 
map $F\colon G\rightarrow G$ is given by 
\[ F=\tilde{\gamma}\circ F_p^m=F_p^m \circ \tilde{\gamma}\qquad
(m\geq 1)\]
where $\tilde{\gamma}\colon G \rightarrow G$ is an automorphism of finite 
order (leaving $T_0,B_0$ invariant) and $F_p\colon G \rightarrow G$ is 
a Frobenius map corresponding to a split $\F_p$-rational structure, such 
that $F_p$ commutes with $\tilde{\gamma}$ and $F_p(t)=t^p$ for all
$t\in T_0$. Thus, $G^F$ is an untwisted or twisted Chevalley group, as in
Steinberg \cite{St}. Note that $\tilde{\gamma}$ induces an automorphism 
of $W$ which is just the automorphism $\gamma\colon W\rightarrow W$ 
induced by $F$ considered earlier. 

\begin{rem} \label{rem20} It is known that all unipotent classes of $G$
are $F_p$-stable (since, in each case, representatives of the classes are
known which lie in $G^{F_p}=G(\F_p)$; see, e.g., Liebeck--Seitz \cite{LiSe}).
Let $C$ be an $F$-stable unipotent class. We shall also make the following
assumption.
\begin{itemize}
\item[($\clubsuit$)] There exists an element $u_0\in C^F$ such that $F$
acts trivially on the finite group of components $A(u_0):=
C_G(u_0)/C_G^\circ(u_0)$.
\end{itemize}
If ($\clubsuit$) holds, then there is a bijective
correspondence between the conjugacy classes of $A(u_0)$ and the conjugacy
classes of $G^F$ that are contained in the set $C^F$ (see, e.g.,
\cite[Lemma~2.12]{LiSe}). For $a\in A(u_0)$, an element in the corresponding
$G^F$-conjugacy class is given by $u_a=hu_0h^{-1}$ where $h\in G$ is such
that $h^{-1}F(h) \in C_G(u_0)$ maps to $a$ under the natural homomorphism
$C_G(u_0)\rightarrow A(u_0)$. (The existence of $h$ is guaranteed by
Lang's Theorem; note that $h$ is not unique but $u_a=hu_0h^{-1}$ is
well-defined up to $G^F$-conjugacy.)
\end{rem}

\begin{rem} \label{lu11} Let $E\in \Irr(W)^\gamma$ and $\iota_G(E)=(C,\cE)
\in \cN_G$, such that $F(C)=C$ and $F^*\cE\cong \cE$. Now let us fix an 
element $u_0\in C^F$ as in ($\clubsuit$), such that
$F$ acts trivially on $A(u_0)$. Then it is known (see Lusztig 
\cite[19.7]{Ldisc4}) that there is a natural $A(u_0)$-module structure on 
the stalk $\cE_{u_0}$; in fact, we have $\cE_{u_0} \in \Irr(A(u_0))$
and there is a root of unity $\delta_E\in\overline{\Q}_\ell$ such that 
\[  Y_E(u_a)= \delta_E\operatorname{Tr}(a, \cE_{u_0})\qquad
\mbox{for all $a\in A(u_0)$}.\]
Since the values of $Y_E$ are integers, it easily follows that
$\delta_E=\pm 1$. (See \cite[Lemma~3.3]{hartur} for further details.)
Note that $\mbox{Tr}(a,\cE_{u_0})$ is just an entry in the ordinary
character table of $A(u_0)$. In particular, if $a=1$, then $u_1$ is
$G^F$-conjugate to $u_0$ and so $\delta_E$ is determined by the identity
\[Y_E(u_0)=\delta_E\dim \cE_{u_0}.\]
\end{rem}

\textit{Thus, the whole problem of computing the Green functions $Q_w$ 
is reduced to the determination of the signs $\delta_E=\pm 1$ for 
$E \in \Irr(W)^\gamma$} (cf.\ Shoji \cite[1.3, p.~161]{S6a}). 

\begin{rem} \label{rem22} In the tables in Carter \cite[\S 13.3]{C2}, 
Lusztig--Spaltenstein \cite{LuSp1} and Spaltenstein \cite{Spa1}, the
pair $(C,\cE)\in \cN_G$ corresponding to $E\in \Irr(W)$ via the Springer 
correspondence is specified by indicating the class $C$, the group 
$A(u)$ (where $u\in C$) and the irreducible $A(u)$-module $\cE_u$. For 
example, in Table~\ref{tabG23}, we have $\cE\cong \overline{\Q}_\ell$ and 
so $\cE_u$ is the trivial representation of $A(u)$, in all cases. For $G$ 
of exceptional type, the possibilities for $A(u)$ are rather limited: 
either $A(u)$ is abelian of order at most~$6$, or a dihedral group of 
order~$8$, or isomorphic to a symmetric group $\fS_r$ where $r=3,4,5$,
or isomorphic to $\Z/2\Z\times \fS_3$ (see \cite[5.4]{Spa1}).
\end{rem}

\begin{rem} \label{critic} Let $C$ be an $F$-stable unipotent class and 
$u_0\in C^F$ be such that $F$ acts trivially on $A(u_0)$. Let $u_0=u_1,u_2,
\ldots, u_r\in C^F$ be representatives of the $G^F$-conjugacy classes that
are contained in $C^F$, and let $a_1,\ldots,a_r \in A(u_0)$ be corresponding 
representatives of the conjugacy classes of $A(u_0)$ (see Remark~\ref{rem20}). 

(a) Let $E_0 \in \Irr(W)^\gamma$ be such that $\iota_G(E_0)=(C,
\overline{\Q}_\ell)$. Then the corresponding $A(u_0)$-module is the trivial
representation and we have 
\[ Y_{E_0}(u_i)=\delta_{E_0} \qquad \mbox{for $1\leq i \leq r$}.\]
Now let $E_1\in \Irr(W)$ be the trivial representation. Then the restriction
of the almost character $R_{E_1}$ to $C^F$ is constant and so 
$p_{E_0,E_1}\delta_{E_0}=1$ (see Example~\ref{sub26a}). 
Hence, the sign $\delta_{E_0}$ is determined by the Lusztig--Shoji algorithm 
in Remark~\ref{sub26}. 

(b) Let $E \in \Irr(W)^\gamma$ be such that $\iota_G(E_0)=(C,\cE)$ where 
$\cE$ is not the trivial local system. Then we have
$Y_{E}(u_i)=\delta_{E}\mbox{Tr}(a_i,\cE_{u_0})$ for $1\leq i \leq r$.
Hence, we obtain 
\[ \lambda_{E_0,E}=\sum_{g \in G^F} Y_{E_0}(g)Y_E(g)=
\delta_{E_0}\delta_E\sum_{1\leq i\leq r} [G^F:C_G(u_i)^F]
\mbox{Tr}(a_i,\cE_{u_0}).\]
The sum on the right hand can be explicitly computed using the knowledge
of the centraliser orders $|C_G(u_i)^F|$ and the character table of the 
group $A(u_0)$. On the other hand, the left hand side is also known 
from the Lusztig--Shoji algorithm in Remark~\ref{sub26}. Hence, if the
left hand side is non-zero, then we also obtain $\delta_{E_0}\delta_E$; 
since $\delta_{E_0}$ is known from (a), this also determines $\delta_E$.

If the left hand side is zero, then some special arguments are required.
A very particular such case occurs for $G$ of type $E_8$ and $p\neq 2,3$, 
where $A(u_0)\cong \fS_3$ and it turns out that $\delta_{E_0}=1$ and 
$\delta_E \equiv q\bmod 3$; see Beynon--Spaltenstein 
\cite[\S 3, Case~5]{BeSp}. (We will encounter a similar case in
Section~\ref{sec7}.)
\end{rem}

\begin{exmp} \label{expau2} Let $C$ be an $F$-stable unipotent class
such that $A(u)\cong \Z/2\Z$ for $u\in C$. Note that $F$ acts trivially 
on $A(u)$ for any $u\in C^F$ (since $A(u)$ has order~$2$). So let us fix 
some $u_0\in C^F$. Assume that $E_0\in \Irr(W)^\gamma$ is such that 
$\iota_G(E_0)=(C,\overline{\Q}_\ell)$. Then the corresponding sign 
$\delta_{E_0}$ is determined as in Remark~\ref{critic}(a). Let us also 
assume that there exists $E\in \Irr(W)^\gamma$ such that $\iota_G(E)=
(C,\cE)$ where $\cE$ is a non-trivial local system. Now $C^F$ splits into 
two classes in $G^F$; let $u_0'\in C^F$ be such that $u_0,u_0'$ are not
conjugate in $G^F$. Then the values of $Y_{E_0}$, $Y_{E}$ are given by 
\[ Y_{E_0}(u_0)=Y_{E_0}(u_0')=\delta_E \qquad\mbox{and}\qquad 
Y_{E}(u_0)=\delta_{E}, Y_{E}(u_0')=-\delta_{E}.\]
(Note that $\overline{\Q}_\ell$ corresponds to the trivial character of
$A(u_0)$ and $\cE$ corresponds to the non-trivial character of $A(u_0)$.)
Now, we either have $\delta_E=\delta_{E_0}$ or $\delta_E=-\delta_{E_0}$.
But, as already discussed in \cite[p.~591]{BeSp}, if we are in the second 
case, then we change the roles of $u_0,u_0'$ and, with the new choice of 
$u_0$, we will have $\delta_E=\delta_{E_0}$. Thus, in the present situation,
we can always choose $u_0\in C^F$ such that $\delta_E=\delta_{E_0}$, and 
$u_0$ is unique up to conjugation within $G^F$. The only remaining problem 
is to identify $u_0$ in a given list of representatives of unipotent classes.
In order to try to solve this problem, we follow Remark~\ref{critic}(b) and 
consider the relation
\begin{equation*}
\lambda_{E_0,E}=\sum_{g \in G^F} Y_{E_0}(g)Y_{E}(g)=|G^F| \bigl(
|C_G(u_0)^F|^{-1}-|C_G(u_0')^F|^{-1}\bigr)\delta_{E_0}\delta_{E}.\tag{$*$}
\end{equation*}
This leads to the following two cases. 

(a) If $\lambda_{E_0,E}\neq 0$, then ($*$) implies $|C_G(u_0)^F|\neq 
|C_G(u_0')^F|$, which distinguishes the representatives $u_0,u_0'$. 

(b) If $\lambda_{E_0,E}=0$, then an additional argument
is required in order to distinguish the representatives~$u_0,u_0'$. (See 
\S \ref{sec42} below for a typical example.)
\end{exmp}

\begin{table}[htbp] \caption{The Springer correspondence and $p_{E',E}$ 
for $G_2$, $p\neq 3$} \label{tabG2} 
\begin{center} $\begin{array}{c@{\hspace{5pt}}c@{\hspace{5pt}}
c@{\hspace{5pt}}c}\hline E & d_E& A(u) &\iota_G(E) \\ \hline
E_{1,6} & 6 & \{1\}& (1,\overline{\Q}_\ell) \\
E_{1,3}'' & 3& \{1\} &(A_1,\overline{\Q}_\ell) \\
E_{2,2} & 2 & \{1\}& (\tilde{A}_1,\overline{\Q}_\ell) \\
E_{2,1} & 1 & \fS_3 &(G_2(a_1),\overline{\Q}_\ell) \\
E_{1,3}' & 1 & \fS_3&  (G_2(a_1),\cE) \\
E_{1,0} & 0& \Z/(p,2)\Z &(G_2,\overline{\Q}_\ell) \\ \hline
\multicolumn{4}{l}{\text{(where $\cE\not\cong \overline{\Q}_\ell$, 
$\dim \cE_u=2$)}}\end{array}\qquad 
\begin{array}{c@{\hspace{5pt}}c@{\hspace{5pt}}c@{\hspace{5pt}}
c@{\hspace{5pt}}c@{\hspace{5pt}}c@{\hspace{5pt}}c}
\hline p_{E',E} & E_{1,6} & E_{1,3}'' & E_{2,2} &
E_{2,1} & E_{1,3}' & E_{1,0}\\\hline
E_{1,6} & 1 & 1 &  q^2{+}1 & q^4{+}1 & q^2 & 1 \\
E_{1,3}'' & 0 & 1 & 1 & 1 & 0 & 1 \\
E_{2,2} & 0 & 0 & 1 & 1 & 1 & 1 \\
E_{2,1} & 0 & 0 & 0 & 1 & 0 & 1\\
E_{1,3}' & 0 & 0 & 0 & 0 & 1 & 0\\
E_{1,0} & 0 & 0 & 0 & 0 & 0 & 1 \\ \hline &&&&&& \end{array}$
\end{center}
\end{table}

\begin{exmp} \label{expg2} Let $G$ be of type $G_2$ and $p\neq 3$. 
Then $F$ acts trivially on $W$ and the induced automorphism $\gamma\colon 
W\rightarrow W$ is the identity. Consequently, $\Irr(W)=\Irr(W)^\gamma$. 
By \cite[Table~22.2.6]{LiSe}, there are $5$ unipotent classes of~$G$, 
which are all $F$-stable. Let $C$ be the unipotent class denoted by 
$G_2(a)$; we have $A(u)\cong \fS_3$ for $u\in C$. The set $C^F$ splits 
into three classes in $G^F$, with centraliser orders $6q^{4},2q^{4},3q^{4}$.
Thus, up to conjugation by elements in $G^F$, there is a unique $u_0\in 
C^F$ such that $|C_G(u_0)^F|=6q^{4}$ and $F$ acts trivially on $A(u_0)$. 
This whole discussion also works for the Frobenius map $F_p$. Thus, we
can even assume that $F_p(u_0)=u_0$ and $F_p$ acts trivially on $A(u_0)$. 
If $C'$ is a unipotent class different from $C$, then $|A(u')|\leq 2$ for 
$u'\in C$. Consequently, condition ($\clubsuit$) holds for all unipotent 
classes of $G$. The Springer correspondence is explicitly described by 
Spaltenstein \cite[p.~329]{Spa1}; see Table~\ref{tabG2}. As in 
Example~\ref{sub27}, we run the function {\tt ICCTable} which yields the 
coefficients $p_{E',E}$. By inspection of Table~\ref{tabG2}, we see that 
there is just one case which is not covered by the arguments in 
Remark~\ref{critic}. The relevant unipotent class is the above-mentioned 
class $C$; we have 
\[ \iota_G(E_{2,1})=(C,\overline{\Q}_\ell)\qquad \mbox{and} \qquad
\iota_G(E_{1,3}')=(C,\cE) \quad\mbox{where} \quad \dim \cE_u=2.\]
Using the output of {\tt ICCTable} and the argument in 
Remark~\ref{critic}(a), we already see that $\delta_{E_{2,1}}=1$.
However, we have $\langle Y_{E_{2,1}},Y_{E_{1,3}'}\rangle=0$ and so we can
not apply the method in Remark~\ref{critic}(b). We now argue as follows.
Using the output of {\tt ICCTable} (see Table~\ref{tabG2}), we compute
the coefficients $\tilde{p}_{E'}$ in Remark~\ref{rem31}. We obtain the 
following formula:
\[ Q_1(u_0)=(2q+1)Y_{E_{2,1}}(u_0)+qY_{E_{1,3}'}(u_0)=(2q+1)+
2q\delta_{E_{1,3}'}.\]
Thus, depending on whether $\delta_{E_{1,3}'}$ equals $+1$ or $-1$, we 
have $Q_1(u_0)=4q+1$ or $Q_1(u_0)=1$. On the other hand, by 
Remark~\ref{rem31}, $Q_1(u_0)$ also is the value at $u_0$ of the character 
of the permutation representation of $G^F$ on the cosets of $B_0^F$. We 
use this interpretation to show that $Q_1(u_0)>1$. For this purpose, it 
will be sufficient to show that $C_G(u_0)^F \not\subseteq B_0^F$. Assume, 
if possible, that $C_G(u_0)^F \subseteq B_0^F$. Then we also have 
$C_G(u_0)^{F_p}\subseteq B_0^{F_p}$. We have a natural homomorphism 
$B_0^{F_p}\rightarrow T_0^{F_p}$, with kernel consisting of unipotent 
elements only. If $p=2$, then $T_0^{F_p}=\{1\}$ and so $C_G(u_0)^{F_p}$ 
would be a unipotent group, contradiction to the fact that $A(u_0)\cong 
\fS_3$ is a quotient of $C_G(u_0)^{F_p}$. If $p\neq 2$ (and $p\neq 3$),
then $A(u_0)\cong \fS_3$ will still be a quotient of the image of 
$C_G(u_0)^{F_p}$ in $T_0^{F_p}$, contradiction since $T_0$ is abelian.
Thus, we do have $C_G(u_0)^F\not\subseteq B_0^F$ and so $Q_1(u_0)>1$, 
as claimed. But this forces $\delta_{E_{1,3}'}=1$.
\end{exmp}

\begin{rem} \label{basep} Assume that $\tilde{\gamma}=\mbox{id}_G$ and
$F=F_p^m$ where $m\geq 1$. Let us first consider the case where $m=1$
and $G^{F_p}=G(\F_p)$ is an untwisted Chevalley group over the prime 
field~$\F_p$. The whole discussion above applies, of course, with $F_p$ 
instead of $F$. In order to have a separate notation from the general 
case, we introduce a superscript ``$\sharp$'' to various objects considered 
earlier. Thus, for $w\in W$, we denote by $R_w^\sharp$ the virtual
representation of $G^{F_p}$ defined by Deligne--Lusztig, and by 
$Q_w^\sharp$ the corresponding Green function. For $E\in\Irr(W)$, let 
$R_E^\sharp\colon G^{F_p}\rightarrow \overline{\Q}_\ell$ be the corresponding 
almost character. (Note that, now, $F_p$ acts trivially on $W$ and so 
$\sigma_E=\mbox{id}_E\colon E\rightarrow E$.) As in Remark~\ref{rem20}, we 
assume that there exists an element $u_0\in C^{F_p}$ such that $F_p$ acts 
trivially on $A(u_0)$. By Remark~\ref{lu11}, there is a well-defined sign,
which we now denote by $\delta_E^\sharp=\pm 1$, such that 
\[ R_E^\sharp(u_0)=\delta_E^\sharp \,p^{d_E}\dim \cE_{u_0}.\]
Having fixed the above notation, we now consider $F=F_p^m$ for any $m\geq 1$. 
Then we still have $F(u_0)=u_0$, and $F$ acts trivially on $A(u_0)$. Let 
$\delta_E=\pm 1$ be as in Remark~\ref{lu11}, now with respect to~$F$. Then, 
by \cite[Theorem~3.7]{hartur}, we have 
\[ \delta_E=(\delta_E^\sharp)^m;\quad\mbox{in particular, $\delta_E=1$
whenever $m$ is even}.\]
Thus, in order to determine $\delta_E$, it is sufficient to consider
the case where $m=1$. This will be our main tool in the discussion 
of groups of exceptional type, in order to deal with those cases which 
are not covered by Remark~\ref{critic}.
\end{rem} 

\section{Explicit realisations of $G(q)$} \label{sec3}

Assume that $G$ is a simple algebraic group. In order to perform explicit
computations on a computer with elements of $G$ (as in the following 
sections), we need a concrete realisation of $G$ as a matrix group. Now, 
in principle it is well-known how to do this, using Chevalley's 
construction as explained in detail by Carter \cite{C1} and Steinberg 
\cite{St}. Computer programs are available as described by 
Cohen--Murray--Taylor \cite{CMT}, for example. Note that the starting point
of this approach is the choice of a Chevalley basis in the corresponding
simple Lie algebra over~$\C$. For example, Mizuno \cite[Table~12]{Miz2} 
explicitly specifies such a choice for type $E_6,E_7,E_8$. But this raises
the following issue. If we want to perform computations with Mizuno's class 
representatives using the Cohen--Murray--Taylor programs, we would first 
need to clarify the relation between the chosen Chevalley bases~---~and 
the same issue arises with any other reference to the literature about 
explicit computations in~$G$. 

Here, we wish to advertise two recent developments with regard to these 
issues. Firstly, Lusztig \cite{L19} gives an explicit, canonical 
construction of $G$ as a matrix group, which does not 
depend at all on the choice of a Chevalley basis. Since this only yields 
root elements in $G$ for simple roots and their negatives, one still
needs to specify a Chevalley basis for further computations. But then, 
secondly, \cite{mylie} produces two canonical choices of a Chevalley 
basis, which differ from each other by a global sign
and, thus, yield ``canonical'' root elements for all roots. The
computer algebra package {\sf ChevLie} \cite{chevlie} implements these 
constructions and works both in {\sf GAP4} \cite{gap} and Michel's version 
of {\sf GAP3} \cite{jmich}. We briefly explain the constructions of 
\cite{mylie}, \cite{L19} and the basic functionality of the {\sf ChevLie} 
package.

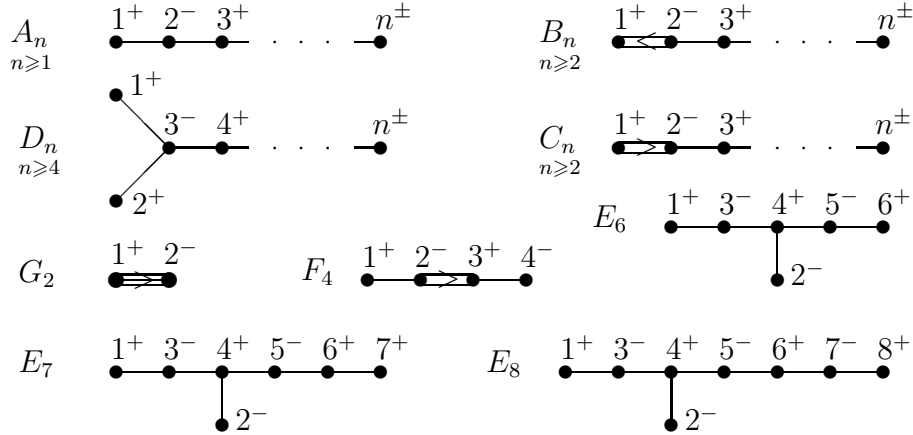
\begin{table}[htbp] \caption{Dynkin diagrams of simple Lie algebras} 
\label{Mdynkintbl} 
\begin{center} 
\makeatletter
\begin{picture}(345,170)
\put( 13, 25){$E_7$}
\put( 50, 25){\circle*{5}}
\put( 48, 30){$1^+$}
\put( 50, 25){\line(1,0){20}}
\put( 70, 25){\circle*{5}}
\put( 68, 30){$3^-$}
\put( 70, 25){\line(1,0){20}}
\put( 90, 25){\circle*{5}}
\put( 88, 30){$4^+$}
\put( 90, 25){\line(0,-1){20}}
\put( 90, 05){\circle*{5}}
\put( 95, 03){$2^-$}
\put( 90, 25){\line(1,0){20}}
\put(110, 25){\circle*{5}}
\put(108, 30){$5^-$}
\put(110, 25){\line(1,0){20}}
\put(130, 25){\circle*{5}}
\put(128, 30){$6^+$}
\put(130, 25){\line(1,0){20}}
\put(150, 25){\circle*{5}}
\put(148, 30){$7^+$}

\put(190, 25){$E_8$}
\put(220, 25){\circle*{5}}
\put(218, 30){$1^+$}
\put(220, 25){\line(1,0){20}}
\put(240, 25){\circle*{5}}
\put(238, 30){$3^-$}
\put(240, 25){\line(1,0){20}}
\put(260, 25){\circle*{5}}
\put(258, 30){$4^+$}
\put(260, 25){\line(0,-1){20}}
\put(260, 05){\circle*{5}}
\put(265, 03){$2^-$}
\put(260, 25){\line(1,0){20}}
\put(280, 25){\circle*{5}}
\put(278, 30){$5^-$}
\put(280, 25){\line(1,0){20}}
\put(300, 25){\circle*{5}}
\put(298, 30){$6^+$}
\put(300, 25){\line(1,0){20}}
\put(320, 25){\circle*{5}}
\put(318, 30){$7^-$}
\put(320, 25){\line(1,0){20}}
\put(340, 25){\circle*{5}}
\put(338, 30){$8^+$}

\put( 13, 59){$G_2$}
\put( 50, 60){\circle*{6}}
\put( 48, 66){$1^+$}
\put( 50, 58){\line(1,0){20}}
\put( 50, 60){\line(1,0){20}}
\put( 50, 62){\line(1,0){20}}
\put( 56, 57){$>$}
\put( 70, 60){\circle*{6}}
\put( 68, 66){$2^-$}

\put(120, 60){$F_4$}
\put(145, 60){\circle*{5}}
\put(143, 65){$1^+$}
\put(145, 60){\line(1,0){20}}
\put(165, 60){\circle*{5}}
\put(163, 65){$2^-$}
\put(165, 58){\line(1,0){20}}
\put(165, 62){\line(1,0){20}}
\put(171, 57.5){$>$}
\put(185, 60){\circle*{5}}
\put(183, 65){$3^+$}
\put(185, 60){\line(1,0){20}}
\put(205, 60){\circle*{5}}
\put(203, 65){$4^-$}

\put(230, 80){$E_6$}
\put(260, 80){\circle*{5}}
\put(258, 85){$1^+$}
\put(260, 80){\line(1,0){20}}
\put(280, 80){\circle*{5}}
\put(278, 85){$3^-$}
\put(280, 80){\line(1,0){20}}
\put(300, 80){\circle*{5}}
\put(298, 85){$4^+$}
\put(300, 80){\line(0,-1){20}}
\put(300, 60){\circle*{5}}
\put(305, 58){$2^-$}
\put(300, 80){\line(1,0){20}}
\put(320, 80){\circle*{5}}
\put(318, 85){$5^-$}
\put(320, 80){\line(1,0){20}}
\put(340, 80){\circle*{5}}
\put(338, 85){$6^+$}

\put( 13,110){$D_n$}
\put( 13,100){$\scriptstyle{n \geq 4}$}
\put( 50,130){\circle*{5}}
\put( 55,130){$1^+$}
\put( 50,130){\line(1,-1){21}}
\put( 50, 90){\circle*{5}}
\put( 56, 85){$2^+$}
\put( 50, 90){\line(1,1){21}}
\put( 70,110){\circle*{5}}
\put( 68,115){$3^-$}
\put( 70,110){\line(1,0){30}}
\put( 90,110){\circle*{5}}
\put( 88,115){$4^+$}
\put(110,110){\circle*{1}}
\put(120,110){\circle*{1}}
\put(130,110){\circle*{1}}
\put(140,110){\line(1,0){10}}
\put(150,110){\circle*{5}}
\put(147,115){$n^\pm$}

\put(210,110){$C_n$}
\put(210,100){$\scriptstyle{n \geq 2}$}
\put(240,110){\circle*{5}}
\put(238,115){$1^+$}
\put(240,108){\line(1,0){20}}
\put(240,112){\line(1,0){20}}
\put(246,107.5){$>$}
\put(260,110){\circle*{5}}
\put(258,115){$2^-$}
\put(260,110){\line(1,0){30}}
\put(280,110){\circle*{5}}
\put(278,115){$3^+$}
\put(300,110){\circle*{1}}
\put(310,110){\circle*{1}}
\put(320,110){\circle*{1}}
\put(330,110){\line(1,0){10}}
\put(340,110){\circle*{5}}
\put(337,115){$n^{\pm}$}

\put( 10,150){$A_n$}
\put( 10,140){$\scriptstyle{n \geq 1}$}
\put( 50,150){\circle*{5}}
\put( 48,155){$1^+$}
\put( 50,150){\line(1,0){20}}
\put( 70,150){\circle*{5}}
\put( 68,155){$2^-$}
\put( 70,150){\line(1,0){30}}
\put( 90,150){\circle*{5}}
\put( 88,155){$3^+$}
\put(110,150){\circle*{1}}
\put(120,150){\circle*{1}}
\put(130,150){\circle*{1}}
\put(140,150){\line(1,0){10}}
\put(150,150){\circle*{5}}
\put(148,155){$n^{\pm}$}

\put(210,150){$B_n$}
\put(210,140){$\scriptstyle{n \geq 2}$}
\put(240,150){\circle*{5}}
\put(238,155){$1^+$}
\put(240,148){\line(1,0){20}}
\put(240,152){\line(1,0){20}}
\put(246,147.5){$<$}
\put(260,150){\circle*{5}}
\put(258,155){$2^-$}
\put(260,150){\line(1,0){30}}
\put(280,150){\circle*{5}}
\put(278,155){$3^+$}
\put(300,150){\circle*{1}}
\put(310,150){\circle*{1}}
\put(320,150){\circle*{1}}
\put(330,150){\line(1,0){10}}
\put(340,150){\circle*{5}}
\put(338,155){$n^\pm$}
\end{picture}
\end{center}
\end{table}

\subsection{Cartan matrices and the $\epsilon$-function} \label{sub41}

Let $I$ be a finite index set and $A=(a_{ij})_{i,j\in I}$ be the Cartan 
matrix of an irreducible (crystallographic) root system. In 
Table~\ref{Mdynkintbl}, we fix a labelling of the corresponding Dynkin 
diagram. Recall that $A$
can be recovered from the diagram as follows. For $i\in I$, we have 
$a_{ii}=2$. Now asssume that $i,j\in I$ are such that $i\neq j$. Then
$a_{ij}=a_{ji}=0$ if $i,j$ are not joined by an edge. We have $a_{ij}=
a_{ji}=1$ if $i,j$ are joined by a simple edge. Furthermore, $a_{ij}=-1$
and $a_{ji}=-2$ if $i,j$ are joined by a double edge with an arrow pointing 
towards~$j$. Finally, $a_{ij}=-1$ and $a_{ji}=-3$ if $i,j$ are joined by 
a triple edge with an arrow pointing towards~$j$. For example:
\[ B_3:\left(\begin{array}{rrr} 2 & -2 & 0 \\ -1 & 2 & -1 \\ 0 & -1 & 2 
\end{array}\right),\qquad
G_2:\left(\begin{array}{rr} 2 & -1 \\ -3 & 2 \end{array}\right).\]
In Table~\ref{Mdynkintbl}, we also specify a function $\epsilon\colon I 
\rightarrow \{\pm 1\}$ such that $\epsilon(i)=-\epsilon(j)$ whenever 
$i\neq j$ and $a_{ij}\neq 0$. Note that, since the diagram is connected, 
there are exactly two such functions: if $\epsilon$ is one of them, then 
the other one is~$-\epsilon$.

\subsection{The Weyl group and the root system} \label{sub42}
Let $V$ be a $\Q$-vector space with a basis $\{\alpha_i \mid i \in I\}$.
For $i \in I$, we define a linear map $s_i\colon V \rightarrow V$ by
$s_i(\alpha_j):=\alpha_j-a_{ij}\alpha_i$ for $j \in I$. Then $s_i^2=
\mbox{id}_V$ and so $s_i \in \text{GL}(V)$. Then the corresponding
Weyl group is given by $W:=\langle s_i \mid i \in I\rangle \subseteq 
\text{GL}(V)$, with root system 
\[ \Phi:=\{w(\alpha_i)\mid i \in I, w \in W\}\subseteq V,\]
where $\{\alpha_i\mid i \in I\}$ is a system of simple roots. In 
{\sf CHEVIE} \cite{chevie}, all of the above is realised by the function
{\tt CoxeterGroup}, which returns a record containing basic data 
corresponding to a given Dynkin diagram. For {\sf ChevLie}, we essentially 
copied the code of that function, so that it works both in {\sf GAP3} 
and {\sf GAP4}. Example:
\begin{verbatim}  
    gap>  W := WeylRecord("B",2);;
    gap>  W.cartan    # the Cartan matrix
    [ [ 2, -2 ], [ -1, 2 ] ]
    gap>  W.roots;    # I-tuples representing the roots
    [ [ 1, 0 ], [ 0, 1 ], [ 1, 1 ], [ 2, 1 ], 
      [ -1, 0 ], [ 0, -1 ], [ -1, -1 ], [ -2, -1 ] ]
    gap>  W.epsilon;
    [ 1, -1 ]
\end{verbatim}
The record component {\tt epsilon} holds the function $\epsilon\colon 
I \rightarrow \{\pm 1\}$. (This is not present in the original {\sf CHEVIE}
system.)

\subsection{The operators $e_i$ and $f_i$} \label{sub43}
For any $\alpha,\beta\in\Phi$ such that $\alpha\neq \pm \beta$, we define
\[p_{\alpha,\beta}:=\max\{i \geq 0 \mid \beta+i\alpha \in \Phi\}\quad
\mbox{and}\quad q_{\alpha,\beta}:=\max\{i \geq 0 \mid \beta-i\alpha \in 
\Phi\}.\]
Thus, $\beta-q_{\alpha,\beta}\alpha,\ldots,\beta-\alpha,\beta,
\beta+\alpha,\ldots, \beta+p_{\alpha,\beta}\alpha$ is the $\alpha$-string
through~$\beta$. Following Lusztig \cite[\S 2]{L19}, we now consider a 
$\Q$-vector space $M$ with a basis $\{u_i \mid i\in I\} \cup
\{v_\alpha\mid \alpha \in \Phi\}$ and define linear maps $e_i\colon 
M\rightarrow M$ and $f_i\colon M\rightarrow M$ by the following formulae, 
where $j\in I$ and $\alpha\in\Phi$.
\begin{alignat*}{2}
e_i(u_j) &:= |a_{ji}|v_{\alpha_i},& \qquad e_i(v_\alpha) 
&:= \left\{\begin{array}{cl} (q_{\alpha_i,\alpha}+1)v_{\alpha+\alpha_i}  
& \mbox{ if $\alpha+\alpha_i \in\Phi$},\\ u_i & \mbox{ if $\alpha=
-\alpha_i$},\\ 0  & \mbox{ otherwise}, \end{array}\right.\\ f_i(u_j) &:= 
|a_{ji}| v_{-\alpha_i},& \qquad f_i(v_\alpha)& :=
\left\{\begin{array}{cl} (p_{\alpha_i,\alpha}+1)v_{\alpha-\alpha_i}  & 
\mbox{ if $\alpha-\alpha_i \in\Phi$},\\ u_i & \mbox{ if $\alpha=
\alpha_i$},\\ 0 & \mbox{ otherwise}. \end{array}\right.
\end{alignat*}
Note that all entries of the matrices of $e_i$, $f_i$ with respect to the 
given basis of $M$ are non-negative integers. We consider $\mbox{End}(M)$ 
as a Lie algebra with the usual Lie bracket $[x,y]:=x \circ y-y\circ x$ 
for $x,y\in\mbox{End}(M)$.  We set $h_i:=[e_i,f_i]$ for $i\in I$. As in 
\cite[\S 4]{mylie}, consider the Lie subalgebra $\fg \subseteq 
\mbox{End}(M)$ generated by $e_i,f_i$ ($i \in I$). Then $\fg$ is a 
(split) simple Lie algebra with Cartan subalgebra $\fh:=\langle 
h_i \mid i \in I\rangle_\Q$ and corresponding root system $\Phi$. In 
particular, we have the Cartan decomposition
\[\fg=\fh\oplus\textstyle{\bigoplus_{\alpha\in \Phi}} \fg_\alpha
\qquad\mbox{where} \qquad \dim \fg_\alpha=1 \mbox{ for all $\alpha\in\Phi$}.\]
In {\sf ChevLie}, we obtain matrices representing $e_i,f_i$ through the 
following command.
\begin{verbatim}
    gap>  W := WeylRecord("E",7);                    
    gap>  r := LieAdjRepresentation(W);;
    gap>  #I dim = 133, Chevalley relations ....... true
\end{verbatim} 
If the basis vectors $\{u_i \mid i\in I\}\cup \{v_\alpha\mid \alpha 
\in \Phi\}$ are ordered as in \cite[Lemma~4.1]{mylie}, then each $e_i$
is a nilpotent upper triangular matrix and each $f_i$ is a nilpotent 
lower triangular matrix. This convention is used in {\tt ChevLie}.
There is also the function {\tt LieMinusculeRepresentation} which 
constructs the matrices in a representation with a minuscule heighest 
weight, as in \cite{mylie1}. 
\begin{verbatim}
    gap>  W := WeylRecord("E",7);                    
    gap>  m := MinusculeWeights(W);
    [ [ 0, 0, 0, 0, 0, 0, 1 ] ]
    gap>  r := LieMinusculeRepresentation(W,m[1]);;
    #I dim = 56, Chevalley relations true
\end{verbatim}

\subsection{Lusztig's construction of Chevalley groups} \label{sub44}

Following Lusztig \cite[\S 2]{L19}, we now obtain a Chevalley group over any
field as follows. Since the $e_i$ and $f_i$ are nilpotent, we can define 
$x_i(t):=\exp(t e_i) \in \mbox{GL}(M)$ and $y_i(t):=\exp(t f_i)\in 
\mbox{GL}(M)$ for all $i\in I$ and $t\in\Q$. Explicitly, we have:
\begin{gather*}
x_i(t)(u_j) = u_j+|a_{ji}|tv_{\alpha_i},\qquad
x_i(t)(v_{-\alpha_i}) =v_{-\alpha_i}+tu_i+t^2v_{\alpha_i},\\
x_i(t)(v_{\alpha_i}) =v_{\alpha_i}, \qquad x_i(t)(v_{\alpha}) =
\sum_{k\geq 0,\, \alpha+k\alpha_i\in\Phi} 
\textstyle{\binom{k+q_{\alpha_i,\alpha}}{k}} t^kv_{\alpha+ k\alpha_i},\\
y_i(t)(u_j) = u_j+|a_{ji}|tv_{-\alpha_i},\qquad
y_i(t)(v_{\alpha_i}) =v_{\alpha_i}+tu_i+t^2v_{-\alpha_i},\\
y_i(t)(v_{-\alpha_i}) =v_{-\alpha_i}, \qquad y_i(t)(v_{\alpha}) =
\sum_{k\geq 0,\,\alpha-k\alpha_i\in\Phi} 
\textstyle{\binom{k+p_{\alpha_i,\alpha}}{k}}
t^kv_{\alpha-k\alpha_i},
\end{gather*}
where $j\in I$ and $\alpha\in \Phi$, $\alpha\neq \pm \alpha_i$. (Compare
with the formulae in \cite[\S 4.3]{C1}.) Now let $K$ be any field
and $\bar{M}$ be a $K$-vector space with a basis $\{\bar{u}_i 
\mid i\in I\} \cup \{\bar{v}_\alpha \mid \alpha\in \Phi\}$. For
$i\in I$ and $t\in K$, we define $\bar{x}_i(t)\in \mbox{GL}(\bar{M})$ 
and $\bar{y}_i(t)\in \mbox{GL}(\bar{M})$ by formulae as above (which 
involve only integer coefficients; see also \cite[\S 4.4]{C1}.) Then
\[G_K:=\langle \bar{x}_i(t), \bar{y}_i(t)\mid i \in I, t\in K \rangle
\subseteq \mbox{GL}(\bar{M})\]
is a Chevalley group over $K$. 

\subsection{The $\epsilon$-canonical Chevalley basis} \label{sub45}
For each $\alpha\in\Phi$, let us choose a non-zero element $e_\alpha
\in\fg_\alpha$. If $\alpha,\beta\in\Phi$ are such that $\alpha+\beta
\in\Phi$, then we define $N_{\alpha,\beta}\in\Q$ by $[e_\alpha,e_\beta]
=N_{\alpha,\beta}e_{\alpha+\beta}$. Now $\{e_\alpha\mid \alpha\in\Phi\}$
is called a \textit{Chevalley basis} if 
\[ N_{\alpha,\beta}=\pm (q_{\alpha,\beta}+1)\quad \mbox{for all 
$\alpha,\beta\in\Phi$ such that $\alpha+\beta \in \Phi$}.\]
Clearly, if $\{e_\alpha\mid \alpha\in\Phi\}$ is a Chevalley basis, then 
so is $\{\pm e_\alpha\mid \alpha\in\Phi\}$, for any choice of the signs. 
Now, having fixed $\epsilon\colon I\rightarrow \{\pm 1\}$, there is a 
unique Chevalley basis $\{\be_\alpha^\epsilon\mid \alpha\in\Phi\}$
such that the following relations hold, for any $i\in I$:
\begin{alignat*}{2} 
\be_{\alpha_i}^\epsilon &= \epsilon(i)e_i, &\quad 
\be_{-\alpha_i}^\epsilon&= -\epsilon(i)f_i,\\
[e_i,\be_\alpha^\epsilon]&=(q_{\alpha_i,\alpha}+1)\be_{\alpha+
\alpha_i}^\epsilon &\qquad &\mbox{if $\alpha+\alpha_i\in \Phi$},\\
[f_i,\be_\alpha^\epsilon]&=(p_{\alpha_i,\alpha}+1)\be_{\alpha-
\alpha_i}^\epsilon &\qquad &\mbox{if $\alpha-\alpha_i\in \Phi$}.
\end{alignat*}
(See \cite[Theorem~5.7 and Example~5.9]{mylie}.)
If we replace $\epsilon$ by $-\epsilon$, then $\be_\alpha^{-\epsilon}=
-\be_\alpha^\epsilon$ for all $\alpha\in\Phi$. In {\sf ChevLie}, the
complete list of elements $\{\be_\alpha^\epsilon\mid \alpha\in\Phi\}$,
as matrices with respect to the basis $\{u_i \mid i\in I\}\cup 
\{v_\alpha\mid \alpha \in \Phi\}$ of $M$, is obtained through the 
command {\tt CanonicalChevalleyBasis(W)}.

\subsection{Root elements} \label{sub46}
Let $\alpha\in\Phi$. By \cite[Cor.~5.6]{mylie}, the linear map 
$\be_\alpha^\epsilon\in\mbox{End}(M)$ is nilpotent and so 
we can define $x_\alpha^\epsilon(t):=\exp(t\be_\alpha^\epsilon)
\in\mbox{GL}(M)$ for any $t\in \Q$. As in \S \ref{sub44}, if $K$ is any 
field, then we obtain analogous elements $\bar{x}_\alpha^\epsilon(t)\in 
G_K$ for $t\in K$. If $\epsilon$ is replaced by $-\epsilon$, then 
$\bar{x}_\alpha^{-\epsilon}(t)= \bar{x}_\alpha^\epsilon(-t)$ for all 
$t\in K$. Thus, having fixed $\epsilon$, we obtain ``canonical'' root 
elements $\bar{x}_\alpha^\epsilon(t) \in G_K$. In {\sf ChevLie}, these 
are obtained as follows.
\begin{verbatim}
    gap>  W := WeylRecord("E",8);                    
    gap>  rep := LieAdjointRepresentation(W);
    gap>  cb := CanonicalChevalleyBasisRep(W,rep); 
    gap>  r := W.roots[70];
    [ 1, 1, 2, 3, 2, 1, 1, 0 ]
    gap>  u := ChevalleyRootElement(W,cb,r,5);   # t=5 
    < matrix 248x248 over the integers >
\end{verbatim} 
(One can equally well use elements from finite fields, of course; 
furthermore, instead of the adjoint representation, one can also use a 
representation with a minuscule highest weight, if such a representation
exists.) Once the elements $\bar{x}_\alpha^\epsilon(t)\in G_K$ are available,
we can also define elements $\bar{h}_\alpha^\epsilon(t)\in G_K$ and 
$\bar{n}_\alpha^\epsilon(t)\in G_K$ by analogous formulae as in 
\cite[Lemma~6.4.4]{C1}. These elements yield the familiar diagonal
elements and lifts of reflections in the Weyl group, respectively. 

In {\sf ChevLie}, the basis of $M$ is ordered such that all 
$\bar{x}_\alpha^\epsilon(t)$, $\alpha\in \Phi^+$, are represented by 
unipotent upper triangular matrices, and all $\bar{x}_\alpha^\epsilon(t)$, 
$\alpha\in \Phi^-$, by unipotent lower triangular matrices; futhermore, all 
$\bar{h}_\alpha^\epsilon(t)$ are diagonal and all 
$\bar{n}_\alpha^\epsilon(t)$ are monomial matrices. Thus, the unipotent 
radical of the standard Borel subgroup of $G_K$, that is, the subgroup
\[ U_K=\langle \bar{x}_\alpha^\epsilon(t) \mid \alpha\in\Phi^+, t\in K
\rangle\subseteq G_K\]
consists precisely of the unipotent upper triangular matrices in $G_K$. 

\subsection{Computing estimates for $Q_1(u)$} \label{sub47}  
Now let $G=G_k$, where $k=\overline{\F}_p$ as in the previous sections,
and $F_p\colon G\rightarrow G$ is given by $F_p(\bar{x}_\alpha^\epsilon(t)) 
=\bar{x}_\alpha^\epsilon(t^p)$ for $\alpha\in \Phi$ and $t\in k$. Once the 
above functions are available, it is straightforward to write a program 
which computes the cardinalities of the sets $Q_{1,w}(u)$ in 
Lemma~\ref{rem31aa}, where $u\in G^{F_p}$. For $w\in W$, let $\Phi_w^+=
\{\beta_1, \ldots,\beta_l\}$ where $l=l(w)$. Then 
\[U_w^{F_p}=\{\bar{x}_{\beta_1}^\epsilon(t_1)\cdots 
\bar{x}_{\beta_l}^\epsilon(t_l)\mid t_i\in \F_p\},\]
with uniqueness of expression. Thus, by running systematically over all 
tuples $(t_1,\ldots,t_l) \in \F_p^l$, we have a way of running through 
the elements of $U_w^{F_p}$, one by one. For each $v\in U_w^{F_p}$, we need 
to check if $v':=\dot{w}^{-1}v^{-1}uv\dot{w} \in B_0^{F_p}$ which, by the 
remarks in \S \ref{sub46}, is simply done by testing if $v'$ is an upper 
triangular matrix. (Some modifications are needed for twisted groups; see 
\S \ref{sec50} below.)

This description shows that computer memory is not an issue, but speed is
critical. We shall have to perform several millions of multiplications of 
matrices (of moderate size) over small finite fields. For this purpose, 
the {\sf GAP} \cite{gap} function {\tt ImmutableMatrix} turns out to be 
particularly efficient. It converts a given matrix into an internal format 
which appears to be highly optimized concerning space and runtime. 

\section{On the Green functions of type $F_4$ in characteristic~$3$}
\label{sec4}

Throughout this section, let $G$ be a simple algebraic group of type 
$F_4$. We have $G=\langle x_\alpha(t)\mid \alpha\in\Phi,t\in k\rangle$
where $\Phi$ is the root system of $G$ with respect to~$T_0$. Let 
$\{\alpha_1,\alpha_2, \alpha_3,\alpha_4\}$ be the set of simple roots 
with respect to~$B$, where the labelling is as in Table~\ref{Mdynkintbl}. 
We assume that $G$ is defined and split over $\F_p$, with corresponding
Frobenius map $F_p\colon G \rightarrow G$ such that $F_p(t)=t^p$ for all 
$t\in T_0$. Let $F=F_p^m$ where $m\geq 1$.  Then
\[ G^F=F_4(q) \qquad \mbox{where} \qquad q=p^m.\]
For $p>3$, the Green functions of $G^F$ have been determined by
Shoji \cite{Sh82}. For $p=2$, the Green functions are explicitly computed
by Malle \cite{Mal1}. It is briefly remarked by Marcelo--Shinoda \cite{MaSh} 
that Shoji's computations remain valid for $p=3$. Since further details
are omitted in \cite{MaSh}, we provide here an independent verification
based on the results in Section~\ref{sec2}; this will also serve as a 
model for the later case studies in Sections~\ref{sec5}--\ref{sec7}. In the 
following, if $\alpha=\sum_{i=1}^4 n_i\alpha_i\in \Phi$, we just write 
$x_{n_1n_2n_3 n_4}(t)$ instead of $x_\alpha(t)$. 

\subsection{Critical unipotent classes for $p=3$} \label{sec41}

Assume from now on that $p=3$. We have $|\Irr(W)|=25$ and the character 
table of $W$ is available in {\sf CHEVIE}. Now $F$ acts trivially on $W$ 
and $\gamma\colon W\rightarrow W$ is the identity. 
Consequently, $\Irr(W)=\Irr(W)^\gamma$. By Shoji \cite{Sh74}, there are $16$
unipotent classes of~$G$, which are all $F_3$-stable. Furthermore, for each 
unipotent class $C$, there exists an element $u_0\in C$ such that $F_3(u_0)
=u_0$ and $F_3$ acts trivially on $A(u_0)$; see \cite[Table~6]{Sh74}. Thus, 
condition ($\clubsuit$) in Section~\ref{sec2} holds. The Springer 
correspondence is explicitly described by Spaltenstein \cite[p.~330]{Spa1}.
As in Example~\ref{sub27}, we run the function {\tt ICCTable} which yields 
the coefficients $p_{E',E}$. By inspection of the output, we see that 
$p_{E',E_1}\in \{0,1\}$ for all $E'\in \Irr(W)$, where $E_1$ is the 
trivial representation of~$W$. Hence, by the argument in 
Remark~\ref{critic}(a), we already have that 
\[ \delta_{E_0}=1\qquad \mbox{for all $E_0\in \Irr(W)$ such that 
$\iota_G(E_0)=(C,\overline{\Q}_\ell)$}. \]
There are further cases which are not covered by the arguments in 
Remark~\ref{critic}(b); these are specified in Table~\ref{tabF4}.
The last two columns specify $E,E_0\in\Irr(W)$ such that $\iota_G(E_0)=
(C,\overline{\Q}_\ell)$ and $\iota_G(E)=(C,\cE)$ with $\cE\not\cong 
\overline{\Q}_\ell$. 

\begin{table}[htbp] \caption{Critical unipotent classes for type $F_4$ with
$p=3$} \label{tabF4} 
\begin{center}
$\begin{array}{c@{\hspace{3pt}}c@{\hspace{5pt}}c@{\hspace{5pt}}
c@{\hspace{5pt}}c@{\hspace{8pt}}c@{\hspace{8pt}}} \hline C & \dim C_G(u) 
& A(u) & |C_G(u)^F| & E_0 & E:\dim \cE_u \\ \hline
F_4(a_1)   &  6 & \Z/2\Z   &  2q^6,2q^6    & \chi_{4,1} &\chi_{2,3}:1\\ \hline
F_4(a_2)   &  8 & \Z/2\Z   &  2q^8,2q^8      & \chi_{9,1} & \chi_{2,1}:1\\ 
\hline
F_4(a_3)   & 12 & \fS_4    &  24q^{12}, 8q^{12},4q^{12},4q^{12},3q^{12} 
& \chi_{12} & \begin{array}{l} \chi_{9,3}: 3\\\chi_{6,2}: 2\\
\chi_{1,3}: 3\end{array}\\  \hline
C_3(a_1)   & 14 & \Z/2\Z   & 2q^{12}(q^2{-}1),2q^{12}(q^2{-}1) & \chi_{16} 
&\chi_{4,3}:1\\ \hline \end{array}$
\end{center}
\end{table}
 
In the table, we use the notation of Spaltenstein \cite{Spa1} for 
$\Irr(W)$. The translation to the notation of Carter \cite[\S 13.2]{C2}
(or {\sf CHEVIE}) is as follows.
\[ \begin{array}{ccccccccc} \hline \chi_{4,1} & \chi_{2,3} & \chi_{9,1} & 
\chi_{2,1} & \chi_{12} & \chi_{9,3} & \chi_{6,2} & \chi_{1,3} & 
\chi_{4,3} \\ \hline \phi_{4,8} & \phi_{2,4}'& \phi_{9,2} & \phi_{2,4}'' & 
\phi_{12,4}& \phi_{9,6}' & \phi_{6,6}'' & \phi_{1,12}' & \phi_{4,7}'' \\ 
\hline\end{array}\]

\subsection{The class $F_4(a_1)$} \label{sec42} Let $C$ be the unipotent 
class denoted by $F_4(a_1)$. Since $A(u)\cong \Z/2\Z$ for $u\in C$, we are 
in the situation of Example~\ref{expau2}, with $E_0=\chi_{4,1}$ and 
$E=\chi_{2,3}$; see Table~\ref{tabF4}. We already know that there exists 
some $u_0\in C^F$ such that $F$ acts trivially on $A(u_0)$ and 
$\delta_{\chi_{4,1}}=\delta_{\chi_{2,3}}=1$. The only remaining problem is to 
identify $u_0$ in a given list of class representatives. For a certain choice
of a Chevalley basis in the Lie algebra of~$G$, a representative $\tilde{u}
\in C$ is explicitly described by Lawther \cite[Table~A]{Law}. Since our
``canonical'' Chevalley basis in Section~\ref{sec3} may be different from
that in \cite{Law}, we can only say that 
\[ \tilde{u}:=x_{1000}(\varepsilon_1)x_{0100}(\varepsilon_2)x_{0110}
(\varepsilon_3) x_{0011}(\varepsilon_4) \in C\]
where $\varepsilon_i=\pm 1$ for $1\leq i \leq 4$. Clearly, we have $\tilde{u}
\in G^{F_3}$. Using our computer programs in \S \ref{sec3}, we check 
explicitly that all elements $\tilde{u}$ as above, for all possible choices 
of the $\varepsilon_i$, are conjugate under elements of $T_0^{F_3}$. Hence, we 
may assume without loss of generality that $\varepsilon_i=1$ for all~$i$. Now 
consider the signs $\delta_{\chi_{4,1}}$ and $\delta_{\chi_{2,3}}$ with 
respect to~$\tilde{u}$. Since $\iota_G(\chi_{4,1})=(C,\overline{\Q}_\ell)$, 
we already know that $\delta_{\chi_{4,1}}=1$. We claim that we also have
$\delta_{\chi_{2,3}}=1$. Since $\tilde{u}\in C^{F_3}$, we can apply 
Remark~\ref{basep}. Thus, it will be sufficient to determine 
$\delta_{\chi_{2,3}}$ in the special case where $m=1$. We now argue as in 
Example~\ref{expg2}. Using the output of {\tt ICCTable}, we find the 
coefficients $\tilde{p}_{E'}$ in \S \ref{rem31}. This yields the formula
\[ Q_1(\tilde{u})=(4q+1)Y_{\chi_{4,1}}(\tilde{u})+2qY_{\chi_{2,3}}(\tilde{u})
=(4q+1) +2q\delta_{\chi_{2,3}}.\]
Setting $q=3$, we obtain $Q_1(\tilde{u})=13+6 \delta_{\chi_{2,3}} \in 
\{19,7\}$. On the other hand, by Remark~\ref{rem31f}, we can try 
to directly compute the value of $Q_1(\tilde{u})$ (or, at least, a lower
bound for that value), by running through the sets $Q_{1,w}(\tilde{u})$ 
and checking if the corresponding coset representatives are 
fixed by~$\tilde{u}$. It turns out that we just need to go up to 
$l(w)\leq 3$ in order to find $19$ cosets that are fixed. (Since we already 
know that $Q_1(\tilde{u})\in\{19,7\}$, it would actually be enough to find 
strictly more than~$7$ cosets that are fixed~---~this simple remark will
be important in later sections when the values of $Q_1$ get significantly
larger.) Thus, we do have $Q_1(\tilde{u})=19$ and, indeed, $\tilde{u}\in C$ 
is a representative such that $\delta_{\chi_{4,1}}=\delta_{\chi_{2,3}}=1$ 
(regardless of the choice of a Chevalley basis).

\subsection{The class $F_4(a_2)$} \label{sec43} Let $C$ be the unipotent 
class denoted by $F_4(a_2)$. Again, we are in the situation of 
Example~\ref{expau2}, now with $E_0=\chi_{9,1}$ and $E=\chi_{2,1}$; see 
Table~\ref{tabF4}. As in the previous case, there exists some $u_0\in C^F$ 
such that $\delta_{\chi_{9,1}}=\delta_{\chi_{2,1}}=1$. The only remaining 
problem is to identify $u_0$ in a given list of class representatives.
By Lawther \cite[Table~A]{Law}, there is a choice of signs $\varepsilon_i=
\pm 1$ such that 
\[ \tilde{u}:=x_{1100}(\varepsilon_1)x_{0120}(\varepsilon_2)x_{0001}
(\varepsilon_3) x_{0011}(\varepsilon_4) \in C. \]
We check again that all elements as above, for all possible choices of 
the signs, are conjugate under elements of $T_0^{F_3}$. Hence, as in the
previous case, we may assume without loss of generality that 
$\varepsilon_i=1$ for all~$i$. We consider the signs
$\delta_{\chi_{9,1}}$ and $\delta_{\chi_{2,1}}$ with respect to~$\tilde{u}$.
Since $\iota_G(\chi_{9,1})=(C,\overline{\Q}_\ell)$, we already know 
that $\delta_{\chi_{9,1}}=1$. Since $\tilde{u}\in C^{F_3}$, it will again be 
sufficient to determine $\delta_{\chi_{2,1}}$ in the special case where 
$m=1$. Using the output of {\tt ICCTable}, we find the formula 
\[ Q_1(\tilde{u})=(9q^2+4q+1)Y_{\chi_{9,1}}(\tilde{u})+2q^2Y_{\chi_{2,1}}
(\tilde{u})=(9q^2+4q+1)+2q^2\delta_{\chi_{2,1}}.\]
Setting $q=3$, we obtain $Q_1(\tilde{u})=94+18 \delta_{\chi_{2,1}} \in 
\{112, 76\}$. Again, by an explicit computation counting coset 
representatives, we find that $Q_1(\tilde{u})=112$. (We just need to look 
at sets $Q_{1,w}(\tilde{u})$ where $l(w)\leq 7$.) Thus, indeed, 
$\tilde{u}\in C$ is a representative such that $\delta_{\chi_{9,1}}=
\delta_{\chi_{2,1}}=1$ (regardless of the choice of a Chevalley basis).

\subsection{The class $F_4(a_3)$} \label{sec44} Let $C$ be the unipotent 
class denoted by $F_4(a_3)$. We have $A(u) \cong \fS_4$ for $u\in C$. By 
Shoji \cite[Table~6]{Sh74}, the set $C^F$ splits into five classes in 
$G^F$, with centraliser orders $24q^{12},4q^{12},8q^{12}, 3q^{12},4q^{12}$. 
Thus, up to conjugation by elements in $G^F$, there is a unique $u_0\in C^F$
such that $|C_G(u_0)^F|=24q^{12}$ and $F$ acts trivially on $A(u_0)$. Now, 
via the Springer correspondence, there are four irreducible representations 
of $W$ associated with $C$. These are $\chi_{12}$, $\chi_{9,3}$, 
$\chi_{6,2}$, $\chi_{1,3}$; see Table~\ref{tabF4}. We already know that 
$\delta_{\chi_{12}}=1$. Now $u_0$ can be chosen to be fixed by~$F_3$; see 
the explicit expression in \cite[Table~6]{Sh74}. So, by 
Remark~\ref{basep}, it is sufficient to determine $\delta_{\chi_{6,2}}$,
$\delta_{\chi_{9,3}}$, $\delta_{\chi_{2,1}}$ in the special case where $m=1$. 

Let $u_0=u_1,u_2,u_3,u_4,u_5\in C^{F_3}$ be representatives of the 
$G^{F_3}$-conjugacy classes that are contained in $C^{F_3}$, and let $a_1,
a_2,a_3,a_4,a_5 \in A(u_0)$ be corresponding representatives of the 
conjugacy classes of $A(u_0)$ (see Remark~\ref{rem20}). 
Using the output of {\tt ICCTable}, and setting $q=3$, we find the formula
\begin{align*}
Q_1(u_i)&= (12q^4+16q^3+9q^2+4q+1)Y_{\chi_{12}}(u_i)\\ &\qquad + (6q^4+4q^3)
Y_{\chi_{6,2}}(u_i)+(9q^4+8q^3+2q^2)Y_{\chi_{9,3}}(u_i)+ 
q^4Y_{\chi_{1,3}}(u_i)\\
&=1498Y_{\chi_{12}}(u_i)+594 Y_{\chi_{6,2}}(u_i)+963
Y_{\chi_{9,3}}(u_i) +81Y_{\chi_{1,3}}(u_i) \qquad\mbox{($q=3$)},
\end{align*}
for $1\leq i \leq 5$. Now, up to the signs $\delta_{\chi_{6,2}}$, 
$\delta_{\chi_{9,3}}$, $\delta_{\chi_{1,3}}$, the values of the 
$Y$-functions on $u_i$ are given by character values of $\fS_4$; see
Remark~\ref{lu11}. Thus, up to those signs, we can explicitly determine 
the values $Q_1(u_i)$. We find that $Q_1(u_i)\leq 5818$ for all~$i$, 
regardless of what the signs $\delta_{\chi_{6,2}}$, $\delta_{\chi_{9,3}}$, 
$\delta_{\chi_{1,3}}$ are; furthermore,
\[ Q_1(u_i)=5818 \qquad \Longleftrightarrow \qquad i=1 \quad\mbox{and}\quad 
\delta_{\chi_{6,2}}= \delta_{\chi_{9,3}}=\delta_{\chi_{1,3}}=1.\]
Hence, if we can find an element $\tilde{u}\in  C^{F_3}$ such that 
$Q_1(\tilde{u})=5818$, then $\tilde{u}$ must be conjugate to $u_0=u_1$ in 
$G^{F_3}$ and $\delta_{\chi_{6,2}}=\delta_{\chi_{9,3}}=\delta_{\chi_{1,3}}=1$.
Now, using the list of representatives in \cite[Table~6]{Sh74} and
adjusting some signs, we consider the element 
\[ \tilde{u}:=x_{1100}(1)x_{0120}(-1)x_{0122}(1)x_{1122}(-1)\in G^{F_3},\] 
where we work with the ``canonical'' Chevalley basis as in Section~\ref{sec3}.
We check that, in the adjoint representation, $\tilde{u}$ has Jordan blocks of 
sizes $7,6^2,5^3, 3^6$. Hence, we have $\tilde{u}\in C$; see 
\cite[Table~4]{Law}. By an explicit computation counting coset 
representatives, we find that $Q_1(\tilde{u})=5818$. (We just need to look 
at sets $Q_{1,w}(\tilde{u})$ where $l(w) \leq 9$.) Thus, indeed, 
$u_0,\tilde{u}$ are conjugate in $G^{F_3}$ and we do have 
$\delta_{\chi_{6,2}}= \delta_{\chi_{9,3}}=\delta_{\chi_{1,3}}=1$.

\subsection{The class $C_3(a_1)$} \label{sec45} Let $C$ be the unipotent 
class denoted by $C_3(a_1)$. Again, we are in the situation
of Example~\ref{expau2}, now with $E_0=\chi_{16}$ and $E=\chi_{4,3}$; see 
Table~\ref{tabF4}. As in the first case, there exists some $u_0\in C^F$ 
such that $\delta_{\chi_{16}}=\delta_{\chi_{4,3}}=1$. The only remaining 
problem is to identify $u_0$ in a given list of class representatives.
By Lawther \cite[Table~A]{Law}, there is a choice of signs $\varepsilon_i=
\pm 1$ such that 
\[x_{0100}(\varepsilon_1)x_{0001}(\varepsilon_2)x_{0120}(\varepsilon_3)\in C
\qquad \mbox{(where $\varepsilon_i=\pm 1$)}.\]
Now we find that all elements as above, for all possible choices of 
the signs, are conjugate under elements of $T_0^{F_3}$ to one of the 
following two elements:
\[\tilde{u}^{\pm}:=x_{0100}(1)x_{0001}(1)x_{0120}(\pm 1).\]
We check that, in the adjoint representation, both $\tilde{u}^{+}$ and 
$\tilde{u}^-$ have Jordan blocks of sizes $7,6^2,5,4^4,3^3,1^3$. Hence, we 
have $\tilde{u}^\pm \in C$; see \cite[Table~4]{Law}. We consider the signs 
$\delta_{\chi_{16}}^\pm$ and $\delta_{\chi_{4,3}}^\pm$ with respect 
to~$\tilde{u}^\pm$. Since $\iota_G(\chi_{16})=(C,\overline{\Q}_\ell)$, we 
already know that $\delta_{\chi_{16}}^\pm=1$. Since $\tilde{u}^{\pm}\in 
C^{F_3}$, it will again be sufficient to determine $\delta_{\chi_{4,3}}^\pm$ 
in the special case where $m=1$. Using the output of {\tt ICCTable}, we 
find the formula 
\begin{align*}
Q_1(\tilde{u}^{\pm})&= (16q^5+36q^4+28q^3+11q^2+4q+1) Y_{\chi_{16}}
(\tilde{u}^{\pm})\\ &\qquad  +(4q^5+10q^4+8q^3+2q^2)Y_{\chi_{4,3}}
(\tilde{u}^{\pm}).
\end{align*}
Setting $q=3$, we obtain $Q_1(\tilde{u}^\pm)=7672+2016\,
\delta_{\chi_{4,3}}^\pm \in\{9688,5656\}$. By an
explicit computation counting coset representatives, we find that
$Q_1(\tilde{u}^\pm)=9688$. (For $\tilde{u}^+$, we just need to look at sets
$Q_{1,w}(\tilde{u}^\pm)$ where $l(w)\leq 13$ in order to find $9688$ 
cosets that are fixed; for $\tilde{u}^-$ we have to go up to $l(w)
\leq 14$ in order to find strictly more than $5656$ cosets that are fixed.)
In particular, this shows that $\tilde{u}^+,\tilde{u}^-$ are conjugate in 
$G^{F_3}$. Hence, indeed, $u=\tilde{u}^+\in C$ is a representative such that 
$\delta_{\chi_{16}}=\delta_{\chi_{4,3}}=1$ (independently of the choice
of a Chevalley basis).

\addtocounter{thm}{5}
\begin{rem} By analogous arguments, we obtain an independent verification 
of the results of Malle \cite{Mal1} on the Green functions of $F_4(2^m)$.
\end{rem}

\section{On the Green functions of untwisted $E_6$ in 
characteristic~$3$} \label{sec5}

Throughout this section, let $G$ be a simple algebraic group of 
(adjoint) type $E_6$. We have $G=\langle x_\alpha(t)\mid \alpha\in\Phi,t\in 
k\rangle$ where $\Phi$ is the root system of $G$ with respect to~$T_0$. Let 
$\{\alpha_i\mid 1\leq i\leq 6\}$ be the set of simple roots with respect
to~$B_0$, where the labelling is chosen as in Table~\ref{Mdynkintbl}. We 
assume that $G$ is defined and split over $\F_p$, with corresponding 
Frobenius map $F_p\colon G \rightarrow G$ such that $F_p(t)=t^p$ for all 
$t\in T_0$. Let $F=F_p^m$ where $m\geq 1$.  Then
\[ G^F=E_6(q) \qquad \mbox{where} \qquad q=p^m.\]
For $p>3$, the Green functions have been determined by Beynon--Spaltenstein 
\cite{BeSp}. For $p=2,3$, the Green functions are explicitly computed by 
Malle \cite{Mal1} and Porsch \cite{Por}. Since the results for $p=3$ 
have never been published, and since we also need them when dealing with 
the twisted case, we will provide here an independent verification of
Porsch's results. 

\subsection{Critical unipotent classes for $p=3$} \label{sec51}

Assume from now on that $p=3$. We have $|\Irr(W)|=25$ and the character 
table of $W$ is available in {\sf CHEVIE}. Now $F$ acts trivially on $W$ 
and $\gamma\colon W\rightarrow W$ is the identity. 
Consequently, $\Irr(W)=\Irr(W)^\gamma$. By Mizuno \cite{Miz}, there are $21$
unipotent classes of~$G$, which are all $F_3$-stable. Furthermore, for each 
unipotent class $C$, there exists an element $u_0\in C$ such that $F_3(u_0)
=u_0$ and $F_3$ acts trivially on $A(u_0)$; see \cite[Prop.~6.1]{Miz}. Thus, 
condition ($\clubsuit$) in Section~\ref{sec2} holds. The Springer 
correspondence is explicitly described by Spaltenstein \cite[p.~331]{Spa1}. 
As in Example~\ref{sub27}, we run the function {\tt ICCTable} which yields 
the coefficients $p_{E',E}$. By inspection of the output, we see that 
there is just one case which is not covered by the arguments in 
Remark~\ref{critic}; see Table~\ref{tabE6} where the last two columns 
specify $E,E_0\in\Irr(W)$ such that $\iota_G(E_0)=(C,\overline{\Q}_\ell)$ 
and $\iota_G(E)=(C,\cE)$ with $\cE\not\cong \overline{\Q}_\ell$. 

\begin{table}[htbp] \caption{The critical unipotent class for
type $E_6$ with $p=3$} \label{tabE6} 
\begin{center}
$\begin{array}{c@{\hspace{10pt}}c@{\hspace{10pt}}c@{\hspace{10pt}}
c@{\hspace{10pt}}c@{\hspace{10pt}}c} \hline C & \dim C_G(u) & A(u) & 
|C_G(u)^F| & E_0 & E \\ \hline
E_6(a_3)   & 12 &\Z/2\Z& 2q^{12}, 2q^{12}      & 30_3 & 15_5\\ 
\hline\end{array}$
\end{center}
\end{table}

In the table, we use the notation of Spaltenstein \cite{Spa1} for 
$\Irr(W)$, which is just a slight variation of Carter \cite[\S 13.2]{C2}
(or {\sf CHEVIE}); for example, the representation $30_3$ is denoted
by $\phi_{30,3}$ in \cite[p.~415]{C2}.

\subsection{The class $E_6(a_3)$} \label{sec52} Let $C$ be the unipotent 
class denoted by $E_6(a_3)$. (Note that Mizuno uses the notation
$A_5{+}A_1$ for this class.) Since $A(u)\cong \Z/2\Z$ for $u\in C$, we 
are in the situation of Example~\ref{expau2}, with $E_0=30_3$ 
and $E=15_5$. The following argument is analogous to that in
\S \ref{sec42}. We already know that there exists some $u_0\in C^F$ 
such that $F$ acts trivially on $A(u_0)$ and $\delta_{30_3}=
\delta_{15_5}$. Using the output of {\tt ICCTable} and the
argument in Remark~\ref{critic}(a), we have $\delta_{30_3}=1$
and, hence, also $\delta_{15_5}=1$. The only remaining problem is 
to identify $u_0$ in a given list of class representatives. By
Mizuno \cite[Lemma~4.3]{Miz}, there is a choice of signs such that
\[x_{18}:=x_{\alpha_5}(\pm 1)x_{\alpha_4}(\pm 1)x_{\alpha_3}(\pm 1)
x_{\alpha_1} (\pm 1) x_{\alpha_6}(\pm 1) x_{\alpha_1{+}\alpha_2{+}
\alpha_3{+}\alpha_4 {+}\alpha_5 {+}\alpha_6}(\pm 1)\in C.\]
But then we check again that all elements as above, for all possible
choices of the signs, are conjugate under elements of $T_0^{F_3}$. Thus,
we may assume that all signs are $+1$. Then we consider the signs 
$\delta_{30_3}$ and $\delta_{15_5}$ with respect to~$x_{18}$. Since 
$\iota_G(30_3)=(C,\overline{\Q}_\ell)$, we already know that 
$\delta_{30_3}=1$. Since $x_{18}\in C^{F_3}$, we can apply 
Remark~\ref{basep}. Thus, it will be sufficient to determine 
$\delta_{15_5}$ in the special case where $m=1$. Using the output of 
{\tt ICCTable}, we find the coefficients $\tilde{p}_{E'}$ in \S \ref{rem31}. 
This yields the formula
\[Q_1(x_{18})=(30q^3+20q^2+6q+1)Y_{30_3}(x_{18})+
(15q^3+6q^2)Y_{15_5}(x_{18}).\]
Setting $q=3$, we obtain $Q_1(x_{18})=1009+459 \delta_{15_5} 
\in \{1468,550\}$. By an explicit computation counting
coset representatives, we find that $Q_1(x_{18})=1468$. (In the
setting of Lemma~\ref{rem31aa}, we just need to look at sets 
$Q_{1,w}(x_{18})$ where $l(w)\leq 12$.) 
Thus, indeed, $u_0:=x_{18}\in C$ is a representative with respect to which 
we have $\delta_{30_3}=\delta_{15_5}=1$ (regardless of the
choice of a Chevalley basis).

\subsection{A different representative for $E_6(a_3)$} \label{sec53} Let 
$C$ be as above, but now consider the element
\[x_{18}':=x_{\alpha_1}(1)x_{\alpha_6}(1)x_{\alpha_3}(1)x_{\alpha_5}(1)
x_{\alpha_4}(1) x_{\alpha_1{+}\alpha_2{+}\alpha_3{+}\alpha_4{+}\alpha_5
{+}\alpha_6}(1)\in G^{F_3}.\]
(This will be useful in the following section, when we consider twisted
groups of type $E_6$.) We claim that $x_{18}$, $x_{18}'$ are conjugate 
in $G^{F_3}$. First we check that, in the adjoint representation, $x_{18}'$ 
has Jordan blocks of sizes $9^4,7,6^4,3^3,2$. Hence, we have $x_{18}'\in C$;
see \cite[Table~6]{Law}. Furthermore, we check that all the elements
\[x_{\alpha_1}(\pm 1)x_{\alpha_6}(\pm 1)x_{\alpha_3}(\pm 1)x_{\alpha_5}
(\pm 1) x_{\alpha_4}(\pm 1) x_{\alpha_1{+}\alpha_2{+}\alpha_3{+}\alpha_4{+}
\alpha_5 {+}\alpha_6}(\pm 1)\]
are conjugate under elements of $T_0^{F^3}$. As above, we set $q=3$ and 
compute that $Q_1(x_{18}')=1468$. Thus, $x_{18}'$ must be conjugate in 
$G^{F_3}$ to~$x_{18}$ (regardless of the choice of a Chevalley basis). 
Furthermore, if we take $x_{18}'$ as the chosen representative in $C^F$, 
then the correponding signs $\delta_{30_3}$ and 
$\delta_{15_5}$ will again be equal to~$1$.

\subsection{Improving efficiency} \label{sec54} 
We have $|W|=51840$ and there are $8335$ elements $w\in W$ such that
$l(w)\leq 12$. So, a priori, in \S \ref{sec52} we would have to look at 
\[ \sum_{w\in W,l(w)\leq 12} 3^{l(w)}=1569060811\]
coset representatives in order to obtain that $Q_1(x_{18})=1468$. Since we 
only need to establish the estimate $Q_1(x_{18}) >550$, we can try to reduce 
the number of elements $w\in W$ to consider, as follows. 

Let $v\in U_w^F$ for some $w\in W$. Writing $v$ as a product of terms 
$x_\alpha(t)$ where $\alpha\in \Phi^+$ and $t\in k$, and using 
Chevalley's commutator relations, we see that 
\begin{align*}
v^{-1}x_{18}v&=x_{\alpha_5}(\pm 1)x_{\alpha_4}(\pm 1)x_{\alpha_3}(\pm 1)
x_{\alpha_1} (\pm 1) x_{\alpha_6}(\pm 1)\\
& \qquad \times \,\mbox{product of terms $x_\alpha(t)$ where 
$\alpha\in\Phi^+$ and $\alpha\neq \alpha_i$ for all~$i$}.
\end{align*}
Now $v$ will only contribute to $|Q_{1,w}(x_{18})|$ if $\dot{w}^{-1}
v^{-1}x_{18} v\dot{w}\in B_0^F$. So we just consider those $w\in W$
such that $l(w)\leq 12$ and the roots $w(\alpha_5)$, $w(\alpha_4)$, 
$w(\alpha_3)$, $w(\alpha_1)$, $w(\alpha_6)$ are all positive. There 
are $47$ such elements $w\in W$, accounting for only $4220491$ cosets. 
It turns out that, already among these cosets, we find more than $550$ ones 
that are fixed by $x_{18}$. (And a similar procedure works in all the other 
cases that we consider in Section~\ref{sec6}, where $|W|=2903040$ and 
such a reduction becomes even more important.) 

\addtocounter{thm}{4}
\begin{rem} By analogous arguments, we obtain an independent
verification of the results of Malle \cite{Mal1} on the Green functions
of $E_6(2^m)$.
\end{rem}

\section{On the Green functions of twisted $E_6$ in 
characteristic~$3$} \label{sec5b}

Throughout this section, let again $G$ be a simple algebraic group of 
(adjoint) type $E_6$. We have $G=\langle x_\alpha(t)\mid \alpha\in\Phi,t\in 
k\rangle$ where $\Phi$ is the root system of $G$ with respect to~$T_0$. Let 
$\{\alpha_i \mid 1\leq i \leq 6\}$ be the set 
of simple roots with respect to~$B_0$, where the labelling is chosen as
in Table~\ref{Mdynkintbl}. We assume that $G$ is defined and split over 
$\F_p$, with corresponding Frobenius map $F_p\colon G \rightarrow G$ such 
that $F_p(t)=t^p$ for all $t\in T_0$. Now we also consider the 
non-trivial graph automorphism $\tilde{\gamma}\colon G\rightarrow 
G$ of order~$2$, such that $\tilde{\gamma}(B_0)=B_0$ and $\tilde{\gamma}
(T_0)=T_0$. This induces the following permutation of the simple roots:
\[ \alpha_1\rightarrow\alpha_6, \quad \alpha_2\rightarrow \alpha_2,
\quad \alpha_3 \rightarrow \alpha_5,\quad \alpha_4\rightarrow\alpha_4,
\quad \alpha_5\rightarrow \alpha_3, \quad \alpha_6\rightarrow \alpha_1.\]
Let $m\geq 1$ and $q=p^m$. Then $G^F={^2\!E}_6(q)$ where 
$F:=\tilde{\gamma}\circ F_p^m=F_p^m\circ \tilde{\gamma}$.

For $p>3$, the Green functions have been determined by 
Beynon--Spaltenstein \cite{BeSp}. For $p=2$, the Green functions are
explicitly computed by Malle \cite{Mal1}. To complete the picture,
it remains to deal with the case $p=3$.

\subsection{Root elements in $G^F$} \label{sec50}
Let $\fg$ be the Lie algebra of type $E_6$, realized as a subalgebra
of $\mbox{End}(M)$ as in Section~\ref{sec3}; recall that $M$ comes
equipped with a basis $\{u_1,\ldots,u_6\}\cup \{v_\alpha\mid
\alpha \in \Phi\}$. Let $\Phi\rightarrow \Phi$, $\alpha\mapsto 
\alpha^\dagger$, be the permutation induced by the automorphism 
$\tilde{\gamma}\colon G \rightarrow G$. Then define a linear map 
$\tau\colon M \rightarrow M$ by 
\[u_1\mapsto u_6,\;\; u_2\mapsto u_2,\;\; u_3 \mapsto u_5,\;\; 
u_4 \mapsto u_4,\;\; u_5 \mapsto u_3,\;\;u_6\mapsto u_1,\; 
\mbox{ and } \;v_\alpha\mapsto v_{\alpha^\dagger}\]
for all $\alpha\in \Phi$; note that $\tau^2=\mbox{id}_M$. Then one 
simply checks that conjugation with $\tau$ inside $\mbox{End}(M)$ 
defines the non-trivial graph automorphism of $\fg$. We shall assume 
throughout this section that $G=G_k\subseteq \mbox{GL}(\bar{M})$ is 
realised as in \S \ref{sub44}. Then $\tilde{\gamma}$ is also realised
by conjugation with $\tau$ inside $\mbox{GL}(\bar{M})$. Furthermore, 
consider the ``canonical'' Chevalley basis $\{\be_\alpha^\epsilon \mid 
\alpha\in \Phi\}$ of $\fg$. Then one also checks that $\tau \circ 
\be_{\alpha}^\epsilon=\be_{\alpha^\dagger}^\epsilon\circ \tau$ for 
all $\alpha\in \Phi$. In this situation, root elements for $G^F$ have a 
simple description as in \cite[Prop.~13.6.3]{C1}, that is, given 
$\alpha\in \Phi$ and $t \in k$, we have  
\begin{align*}
x_{\alpha}(t)\in G^F &\qquad \mbox{if $\alpha^\dagger=\alpha$ 
and $t^q=t$},\\
x_{\alpha}(t)x_{\alpha^\dagger}(t^q) \in G^F & \qquad\mbox{if 
$\alpha^\dagger\neq \alpha$ and $t^{q^2}=t$}.
\end{align*}
(Note that, in type $E_6$, we have $\alpha+\alpha^\dagger
\not\in \Phi$ for all $\alpha\in \Phi$; so we only have to
consider cases (i) and (ii) of \cite[Prop.~13.6.3]{C1}.) It is
then straightforward to adjust the program in \S \ref{sub47} to
the present situation.

\subsection{Critical unipotent classes for $p=3$} \label{sec51b}

Assume from now on that $p=3$. The induced automorphism $\gamma\colon 
W\rightarrow W$ is given by conjugation with the longest element 
$w_0\in W$. Consequently, $\Irr(W)= \Irr(W)^\gamma$. For each 
$E\in \Irr(W)$, we need to choose a map $\sigma_E \colon E\rightarrow E$
as in \S \ref{sub21}. In {\sf CHEVIE}, the ``preferred'' choice for 
$\sigma_E$ specified by Lusztig \cite[17.2]{L2d} is taken. By 
\cite[Lemma~20.16]{LiSe}, all the $21$ unipotent classes of $G$ 
are stable under $F_3$ and under $\tilde{\gamma}$. Further information 
about the classes is provided in \cite[Table~22.2.3]{LiSe}. This shows 
that, for each unipotent class $C$, there exists some $u_0\in C^F$ such 
that $F$ acts trivially on $A(u_0)$. Thus, condition ($\clubsuit$) in 
Section~\ref{sec2} holds. As in Example~\ref{sub27}, we run the function 
{\tt ICCTable} which yields the coefficients $p_{E',E}$: 

\begin{verbatim}
    gap>  W := RootDatum("2E6");; 
    gap>  Display(CharTable(W)); 
    gap>  uc := UnipotentClasses(W,3);;  # p=3
    gap>  Display(uc);  Display(ICCTable(uc));
\end{verbatim}

By inspection of the output, we see that there is only one case which is 
not covered by the arguments in Remark~\ref{critic}, exactly as in
Section~\ref{sec5}, Table~\ref{tabE6}.

\subsection{The class $E_6(a_3)$ (twisted case)} \label{sec52b} Let $C$ be 
the unipotent class denoted by $E_6(a_3)$. Since $A(u)\cong \Z/2\Z$ for 
$u\in C$, we are in the situation of Example~\ref{expau2}, with 
$E_0=30_3$ and $E=15_5$. The following argument is analogous 
to that in \S \ref{sec52}, but some additional care is needed because
of the presence of the graph automorphism~$\tilde{\gamma}$. We already know 
that there exists some $u_0\in C^F$ such that $F$ acts trivially on $A(u_0)$ 
and $\delta_{30_3}=\delta_{15_5}$. Using the output of 
{\tt ICCTable} and the argument in Remark~\ref{critic}(a), we have 
$\delta_{30_3}=1$ and, hence, also $\delta_{15_5}=1$. Again,
the only remaining problem is to identify $u_0$ in a given list of class 
representatives. We claim that, regardless of the choice of a Chevalley 
basis in the Lie algebra of~$G$, we can take $u_0$ to be the element
already considered in \S \ref{sec53}:
\[x_{18}':=x_{\alpha_1}(1)x_{\alpha_6}(1)x_{\alpha_3}(1)x_{\alpha_5}(1)
x_{\alpha_4}(1) x_{\alpha_1{+}\alpha_2{+}\alpha_3{+}\alpha_4{+}\alpha_5
{+}\alpha_6}(1)\in C.\]
By \S \ref{sec50}, we have $x_{18}'\in G^F$; in fact, $x_{18}'$
is fixed by both $F_3$ and~$\tilde{\gamma}$. Conjugating by elements in
$T_0^F$, one sees again that the $G^F$-conjugacy class of $x_{18}'$
is well-defined, regardless of the choice of a Chevalley basis. We 
consider the signs $\delta_{30_3}$ and $\delta_{15_5}$ with
respect to~$x_{18}'$. Since $\iota_G(30_3)=(C,\overline{\Q}_\ell)$, we 
already know that $\delta_{30_3}=1$. So it remains to show that 
$\delta_{15_5}=1$. Since $x_{18}'$ is fixed by $F_3$ and by
$\tilde{\gamma}$, we can apply the argument in \cite[Remark~3.8]{hartur}.
This shows that 
\[\delta_{15_5}\delta_{15_5}^\circ\quad\mbox{does not depend
on $m$},\]
where $\delta_{15_5}^\circ$ is the sign from the untwisted
case in Section~\ref{sec5}. By \S \ref{sec53}, we have 
$\delta_{15_5}^\circ=1$ for all~$m$. Hence, we conclude that
$\delta_{15_5}$ does not depend on~$m$ either. So it will be 
sufficient to determine $\delta_{15_5}$ in the special case where 
$m=1$. Using the output of {\tt ICCTable}, we find the coefficients 
$\tilde{p}_{E'}$ in \S \ref{rem31}. This yields the formula
\[ Q_1(x_{18}')=(10q^3+4q^2+2q+1)Y_{30_3}(x_{18}')+(q^3-2q^2)
Y_{15_5}(x_{18}').\]
Setting $q=3$, we obtain $Q_1(x_{18}')=313+9\delta_{15_5} \in 
\{322,304\}$. By an explicit computation counting coset representatives, 
we find $Q_1(x_{18}')=322$. (In the setting of Lemma~\ref{rem31aa}, we just 
need to look at sets $Q_{1,w}(x_{18}')$ where $l(w)\leq 12$.) Thus, indeed, 
$x_{18}'\in C$ is a representative with respect to which we have 
$\delta_{30_3}=\delta_{15_5}=1$.

\addtocounter{thm}{2}
\begin{rem} By analogous arguments, we obtain an independent
verification of the results of Malle \cite{Mal1} on the Green functions
of ${^2\!E}_6(2^m)$.
\end{rem}

\section{On the Green functions of type $E_7$ in characteristics~$2,3$}
\label{sec6}

Throughout this section, let $G$ be a simple algebraic group of 
(adjoint) type $E_7$. We have $G=\langle x_\alpha(t)\mid \alpha\in\Phi,t\in 
k\rangle$ where $\Phi$ is the root system of $G$ with respect to~$T_0$. Let 
$\{\alpha_i \mid 1\leq i\leq 7\}$ be 
the set of simple roots with respect to~$B_0$, where the labelling is chosen 
as in Table~\ref{Mdynkintbl}. We assume that $G$ is defined and split over 
$\F_p$, with corresponding Frobenius map $F_p\colon G \rightarrow G$ such 
that $F_p(t)=t^p$ for all $t\in T_0$. Let $F=F_p^m$ where $m\geq 1$.  Then
\[ G^F=E_7(q) \qquad \mbox{where} \qquad q=p^m.\]
For $p>3$, the Green functions have been determined by Beynon--Spaltenstein 
\cite{BeSp}. To complete the picture, it remains to deal with the cases 
$p=2,3$. In the following, if $\alpha=\sum_{i=1}^7 n_i\alpha_i\in \Phi$, 
we just write $x_{n_1n_2\ldots n_7}(t)$ instead of $x_\alpha(t)$. 

\subsection{Critical unipotent classes for $p=2,3$} \label{sec61}

Assume from now on that $p=2$ or $p=3$. We have $|\Irr(W)|=60$ and the
character table of $W$ is available in {\sf CHEVIE}. Now $F$ acts trivially 
on $W$ and $\gamma\colon W\rightarrow W$ is the 
identity.  Consequently, $\Irr(W)=\Irr(W)^\gamma$. The unipotent classes
of $G$ have been classified by Mizuno \cite{Miz2}. Each unipotent class $C$ 
is $F$-stable and there exists an element $u_0\in C$ such that $F_p(u_0)=
u_0$ and $F_p$ acts trivially on $A(u_0)$; see \cite[Table~2]{Miz2}. Thus, 
condition ($\clubsuit$) in Section~\ref{sec2} holds. The Springer 
correspondence is explicitly described by Spaltenstein
\cite[p.~331--333]{Spa1}. As in 
Example~\ref{sub27}, we run the function {\tt ICCTable} which yields the 
coefficients $p_{E',E}$. By inspection of the output, we see that $p_{E',
E_1}\in \{0,1\}$ for all $E'\in \Irr(W)$, where $E_1$ is the trivial 
representation of $W$. Hence, by Remark~\ref{critic}(a), we already have 
that 
\[ \delta_{E_0}=1\qquad \mbox{for all $E_0\in \Irr(W)$ such that 
$\iota_G(E_0)=(C,\overline{\Q}_\ell)$}. \]
There are further cases which are not covered by the arguments in
Remark~\ref{critic}(b); these are specified in Table~\ref{tabE7} where, 
as before, the last two columns specify $E,E_0\in\Irr(W)$ such that 
$\iota_G(E_0)=(C,\overline{\Q}_\ell)$ and $\iota_G(E)=(C,\cE)$ with 
$\cE\not\cong \overline{\Q}_\ell$. 

\begin{table}[htbp] \caption{Critical unipotent classes for
type $E_7$ with $p=2,3$} \label{tabE7} 
\begin{center}
$\begin{array}{c@{\hspace{9pt}}c@{\hspace{9pt}}c@{\hspace{9pt}}
c@{\hspace{9pt}}c@{\hspace{9pt}}c@{\hspace{9pt}}c} \hline p & 
C & \dim C_G(u) & A(u) & |C_G(u)^F| & E_0 & E:\dim \cE_u\\ \hline
2,3 & E_7(a_3)   & 13 &\Z/2\Z& 2q^{13}, 2q^{13}      & 56_3 & 21_6\\  \hline
3 & E_7(a_4)   & 17 &\Z/2\Z& 2q^{17}, 2q^{17}      & 189_5 & 15_7\\  \hline
2,3 & E_7(a_5)   & 21 &\fS_3& 6q^{21}, 2q^{21}, 3q^{21}  & 315_7 & 
\begin{array}{c} 280_9:2\\ 35_{13}:1 \end{array}\\\hline 
2,3 & E_6(a_3)   & 23 &\Z/2\Z& 2q^{21}(q^2{-}1), 2q^{21}(q^2{-}1) & 405_8
& 189_{10}\\ \hline
\end{array}$
\end{center}
\end{table}

In the table, we use the notation of Spaltenstein \cite{Spa1} for 
$\Irr(W)$, which is just a slight variation of Carter \cite[\S 13.2]{C2}
(or {\sf CHEVIE}); for example, the representation $56_3$ is denoted
by $\phi_{56,3}$ in \cite[p.~416]{C2}.

\subsection{The class $E_7(a_3)$ for $p=2,3$} \label{sec62} Let $C$ 
be the unipotent class denoted by $E_7(a_3)$. (Note that Mizuno uses the 
notation $D_6{+}A_1$ for this class.) Since $A(u)\cong \Z/2\Z$ for 
$u\in C$, we are in the situation of Example~\ref{expau2}, with 
$E_0=56_3$ and $E=21_6$. The following argument is analogous to that in
\S \ref{sec45}. We already know that there exists some $u_0\in C^F$ such 
that $F$ acts trivially on $A(u_0)$ and $\delta_{56_3}=\delta_{21_6}=1$.
The only remaining problem is to identify $u_0$ in a given list of class 
representatives. Using Mizuno \cite[Table~2]{Miz}, checking sizes of
Jordan blocks, and arguing as in \S \ref{sec45}, we may take
\begin{align*}
\tilde{u}^\pm& :=x_{1000000}(1)x_{0101000}(1)x_{0011000}(1)
x_{0101100}(1)\\ & \qquad \cdot x_{0111100}(\pm 1)x_{0000110}(1)
x_{0000001}(1) \in C,
\end{align*}
regardless of the sign or the choice of a Chevalley basis.  We consider
the signs $\delta_{56_3}^\pm$ and $\delta_{21_6}^\pm$ with respect to 
$\tilde{u}^\pm$. Since $\iota_G(56_3)=(C,\overline{\Q}_\ell)$, we already 
know that $\delta_{56_3}^\pm=1$. Since $\tilde{u}^\pm\in C^{F_p}$, we can 
apply Remark~\ref{basep}. Thus, it will be sufficient to determine 
$\delta_{21_6}^\pm$ in the special case where $m=1$. Using the output
of {\tt ICCTable}, and setting $q=2$ or $q=3$, we obtain the formulae
\begin{align*}
Q_1(\tilde{u}^\pm)&=(56q^3+27q^2+7q+1)Y_{56_3}(\tilde{u}^\pm)+
(21q^3+7q^2)Y_{21_6}(\tilde{u}^\pm)\\ &=\left\{\begin{array}{cl}
571+196\delta_{21_6} \in \{767,375\} & \qquad \mbox{($q=2$)},\\ 
1777+630\delta_{21_6}\in \{2407,1147\} & \qquad \mbox{($q=3$)}.
\end{array}\right.
\end{align*}
(Of course, in general, the polyonomial expressions for the values of $Q_1$
will depend on whether $p=2$ or $p=3$, but for the classes in 
Table~\ref{tabE7}, they do coincide.) By an explicit computation counting 
coset representatives (see Lemma~\ref{rem31aa}), we find that 
$Q_1(\tilde{u}^\pm)$ equals $767$ if $p=2$, and $2407$ if $p=3$. 
(If $p=2$, then we just need to look at sets $Q_{1,w}(\tilde{u}^\pm)$ 
where $l(w)\leq 15$ in order to find $767$ cosets that are fixed; 
if $p=3$, then we just need to go up to $l(w)\leq 6$ in order to find 
strictly more than $1147$ cosets that are fixed.) In particular, 
$\tilde{u}^+$, $\tilde{u}^-$ are conjugate in $G^F$. Thus, indeed,
$\tilde{u}^+\in C$ is a representative with respect to which we have
$\delta_{56_3}=\delta_{21_6}=1$ (regardless of the choice of a Chevalley
basis).

\subsection{The class $E_7(a_4)$ for $p=3$} \label{sec66} Let $p=3$ and $C$ 
be the unipotent class denoted by $E_7(a_4)$. (Note that Mizuno uses the 
notation $D_6(a_1){+}A_1$ for this class.) Since $A(u)\cong \Z/2\Z$ for 
$u\in C$, we are in the situation of Example~\ref{expau2}, with 
$E_0=189_5$ and $E=15_7$. The following argument is analogous to that in
\S \ref{sec45}. We already know that there exists some $u_0\in C^F$ 
such that $F$ acts trivially on $A(u_0)$ and $\delta_{189_7}=
\delta_{15_7}=1$. The only remaining problem is to identify 
$u_0$ in a given list of class representatives. Using Mizuno 
\cite[Table~2]{Miz}, checking sizes of Jordan blocks, and arguing as 
in \S \ref{sec45}, we may take
\begin{align*}
\tilde{u}^\pm& :=x_{1000000}(1)x_{0111000}(1)x_{0011100}(1)x_{0101100}(1)\\
&\qquad\cdot x_{0001110}(1)x_{0011110}(\pm 1)x_{0000011}(1)\in C,
\end{align*}
regardless of the sign or the choice of a Chevalley basis.  We consider
the signs $\delta_{189_5}^\pm$ and $\delta_{15_7}^\pm$ with respect
to $\tilde{u}^\pm$. We already know that $\delta_{189_5}^\pm=1$. 
Since $\tilde{u}^\pm\in C^{F_3}$, we can apply Remark~\ref{basep}. Thus, it 
will be sufficient to determine $\delta_{15_7}$ in the special case 
where $m=1$. Using the output of {\tt ICCTable}, we find the following
formula.
\[Q_1(\tilde{u}^\pm)= (189q^5+155q^4+77q^3+27q^2+7q+1)Y_{189_5}
(\tilde{u}^\pm)+ 15q^5Y_{15_7}(\tilde{u}^\pm).\]
Setting $q=3$, we obtain that $Q_1(\tilde{u}^\pm)=60826+3645
\delta_{15_7}\in \{64471,57181\}$. By an 
explicit computation counting coset representatives, we find that
$Q_1(\tilde{u}^\pm)=64471$. (We just need to look at sets
$Q_{1,w}(\tilde{u}^\pm)$ where $l(w) \leq 16$ in order to find 
strictly more than $57181$ cosets that are fixed by~$\tilde{u}^\pm$.) 
Thus, $\tilde{u}^+$ and $\tilde{u}^-$ are conjugate in $G^F$ and,
indeed, $\tilde{u}^+\in C$ is a representative with respect to which
we have $\delta_{15_7}=1$ (regardless of the choice of a Chevalley basis).

\subsection{The class $E_7(a_5)$ for $p=2,3$} \label{sec64} Let $C$ 
be the unipotent class denoted by $E_7(a_5)$. (Note that Mizuno uses the 
notation $D_6(a_2){+}A_1$ for this class.) We have $A(u) \cong \fS_3$ for 
$u\in C$. The set $C^F$ splits into three classes in $G^F$, with centraliser 
orders $6q^{21},2q^{21},3q^{21}$. Thus, up to conjugation by elements in 
$G^F$, there is a unique $u_0\in C^F$ such that $|C_G(u_0)^F|=6q^{21}$ and $F$ 
acts trivially on $A(u_0)$. Now, via the Springer correspondence, there are
 three irreducible representations of $W$ associated with $C$. These are 
$315_7$, $280_9$, $35_{13}$; see Table~\ref{tabE7}. We already know that 
$\delta_{315_7}=1$. Now $u_0$ can be chosen to be fixed by $F_p$; see the 
explicit expression in \cite[Table~2]{Miz2}. So, by Remark~\ref{basep}, 
it is sufficient to determine $\delta_{280_9}$, $\delta_{35_{13}}$ in the
special case where $m=1$. The following argument is analogous to that in 
\S \ref{sec44}.

Let $u_0=u_1,u_2,u_3\in C^{F_p}$ be representatives of the 
$G^{F_p}$-conjugacy classes that are contained in $C^{F_p}$, and let $a_1,
a_2,a_3\in A(u_0)$ be corresponding representatives of the conjugacy 
classes of $A(u_0)$, where the notation is such that $a_2$ corresponds to a 
transposition in $\fS_3\cong A(u_0)$ and $a_3$ to a $3$-cycle. Using the 
output of {\tt ICCTable}, and setting $q=2$ or $q=3$, we find the formula
\begin{align*}
Q_1(u_i)&= (315q^7+483q^6+350q^5+182q^4+77q^3+27q^2+7q+1)Y_{315_7}(u_i)\\
&\;+ (280q^7+330q^6+161q^5+48q^4+7q^3) Y_{280_9}(u_i)+(35q^7+21q^6)
Y_{35_{13}}(u_i)\\
& =\left\{\begin{array}{cl} 86083\,Y_{315_7}(u_i)+ 62936\,Y_{280_9}(u_i)+ 
5824\,Y_{35_{13}}(u_i) & \qquad\mbox{($q=2$)},\\ 1143148\,Y_{315_7}(u_i)+ 
896130\,Y_{280_9}(u_i)+ 91854\,Y_{35_{13}}(u_i) & \qquad\mbox{($q=3$)},
\end{array}\right.
\end{align*}
for $1\leq i \leq 3$. Now, up to the signs $\delta_{280_9}$ and  
$\delta_{35_{13}}$, the values of the $Y$-functions on $u_i$ are given 
by character values of $\fS_3$; see Remark~\ref{lu11}. Thus, we obtain
\[Q_1(u_1)=\left\{\begin{array}{cl} 86083+ 2\cdot 62936 \delta_{280_9}+ 5824
\delta _{35_{13}} &\qquad\mbox{($q=2$)},\\ 1143148+ 2\cdot 896130
\delta_{280_9}+ 91854\delta_{35_{13}} & \qquad\mbox{($q=3$)}.\end{array}
\right.\]
Since $Q_1(u_1)\geq 0$, this already forces that $\delta_{280_9}=1$ in
both cases. We claim that we also have $\delta_{35_{13}}=1$. For this
purpose, we consider the values at $u_3$:
\[ Q_1(u_3) =\left\{\begin{array}{cl} 86083-62936 + 5824\delta _{35_{13}}
=23147+ 5824\delta _{35_{13}} & \quad\mbox{($q=2$)},\\ 1143148-896130+ 
91854\delta_{35_{13}}=247018+ 91854\delta_{35_{13}} &\quad\mbox{($q=3$)}.
\end{array}\right.\]
By Mizuno \cite[Lemma~21]{Miz2}, there is a choice of signs $\varepsilon_i=
\pm 1$ such that 
\begin{align*}
y_{46}&:=x_{1011000}(\varepsilon_1)x_{0111000}(\varepsilon_2)
x_{0011100}(\varepsilon_3)x_{0101100}(\varepsilon_4)\\&
\qquad \cdot x_{0001110}(\varepsilon_5)x_{0000111}(\varepsilon_6)
x_{1111100}(\varepsilon_7)x_{1111110}(\varepsilon_8)\in C
\end{align*}
and $|C_G(y_{46})^F|=3q^{21}$. Then $y_{46}$ will be conjugate to 
$u_3$ in~$G^F$. 

If $q=2$, then $\varepsilon_i=1$ for all~$i$, and the above formula shows 
that $Q_1(u_3)\in \{28971,17323\}$. By an explicit computation counting 
coset representatives (see Lemma~\ref{rem31aa}), we find that 
$Q_1(y_{46})=17323$. (We just need to look at sets $Q_{1,w}(y_{46})$
where $l(w)\leq 13$.) Hence, $\delta_{35_{13}}=1$, as claimed.

Now assume that $p=3$. Then $Q_1(u_3)\in \{338872,155164\}$. Now we 
simply consider all elements $y_{46}$ as above, for all posssible choices 
of the signs $\varepsilon_i$. For each such choice, $y_{46}$ has Jordan
blocks of sizes $9^9,7,6^2,53,3^6$ and, hence, $y_{46}\in C$; see
\cite[Table~8]{Law}. Furthermore, we find that $Q_1(y_{46})=338872$ in each
case. (We just need to look at sets $Q_{1,w}(y_{46})$ where $l(w)\leq 8$.)
Hence, $\delta_{35_{13}}=1$, as claimed.

\subsection{The class $E_6(a_3)$ for $p=2,3$} \label{sec65} Let $C$ 
be the unipotent class denoted by $E_6(a_3)$. (Note that Mizuno uses the 
notation $(A_5{+}A_1)'$ for this class.) Since $A(u)\cong \Z/2\Z$ for 
$u\in C$, we are in the situation of Example~\ref{expau2}, with 
$E_0=405_8$ and $E=189_{10}$. The following argument is analogous to that in
\S \ref{sec42}. We already know that there exists some $u_0\in C^F$ 
such that $F$ acts trivially on $A(u_0)$ and $\delta_{405_8}=
\delta_{189_{10}}=1$. The only remaining problem is to identify 
$u_0$ in a given list of class representatives. Using Mizuno 
\cite[Table~2]{Miz}, checking sizes of Jordan blocks, and arguing as 
in \S \ref{sec42}, we may take
\[\tilde{u}:=x_{1010000}(1)x_{0111000}(1)x_{0011100}(1)x_{1111000}(1)
x_{0101110}(1)x_{0001111}(1)\in C,\]
regardless of the choice of a Chevalley basis. We consider the signs 
$\delta_{405_8}$ and $\delta_{189_{10}}$ with respect to~$\tilde{u}$. 
We already know that $\delta_{405_8}=1$. Since $\tilde{u}\in C^{F_3}$, 
we can apply Remark~\ref{basep}. Thus, it will be sufficient to 
determine $\delta_{189_{10}}$ in the special case where $m=1$. Using 
the output of {\tt ICCTable}, we find the following formula.
\begin{align*}
Q_1(\tilde{u})&=(405q^8+973q^7+933q^6+532q^5+230q^4+84q^3+27q^2+7q+1)
Y_{405_8}(\tilde{u})\\ &\qquad\qquad +(189q^8+420q^7+351q^6+161q^5+
48q^4+7q^3)Y_{189_{10}}(\tilde{u})\\ 
&= \left\{\begin{array}{cl} 309435+ 130584\delta_{189_{10}}\in 
\{440019,178851\} &\qquad \mbox{($q=2$)},\\ 5615752 + 2457648 
\delta_{189_{10}} \in \{8073400,3158104\}&\qquad \mbox{($q=3$)}.
\end{array}\right.
\end{align*} 
On the other hand, using Lemma~\ref{rem31aa}, we can try to directly
compute $Q_1(\tilde{u})$. 

If $p=2$, then we just need to look at sets $Q_{1,w}(\tilde{u})$
where $l(w)\leq 12$ in order to find strictly more than $178851$ cosets 
that are fixed. A comparison with the above formula shows that we must have 
$\delta_{189_{10}}=1$. If $p=3$, then we need to go up to $l(w)\leq
13$ in order to find strictly more than $3158104$ cosets that are fixed. 
(This is the hardest case for type $E_7$; the computation requires
less than $4$~GB of main memory but takes about a week on a standard 
computer, even with efficiency improvements as in \S \ref{sec54}.) 
So we also have $\delta_{189_{10}}=1$ in this case.

\section{On the Green functions of type $E_8$ in characteristics~$2$}
\label{sec7}

In this final section, let $G$ be a simple algebraic group of
type $E_8$ and $F_p\colon G \rightarrow G$ be a split Froebnius map
as before. Let $F=F_p^m$ where $m\geq 1$. Then 
\[ G^F=E_8(q) \qquad \mbox{where} \qquad q=p^m.\]
For $p>5$, the Green functions have been determined by Beynon--Spaltenstein
\cite{BeSp}. Here, we will not be able to complete the computation of the 
Green functions for the cases where $p=2,3,5$. But we can at least show 
that the very particular case mentioned in Remark~\ref{critic} also 
occurs for $p=2$. So assume from now on that $p=2$.

We have $|\Irr(W)|=112$ and the character table of $W$ is available in 
{\sf CHEVIE}. Now $F$ acts trivially on $W$ and $\gamma\colon W\rightarrow 
W$ is the identity. Consequently, $\Irr(W)=\Irr(W)^\gamma$. The unipotent 
classes of $G$ have been classified by Mizuno \cite{Miz2}. The Springer 
correspondence is explicitly described by Spaltenstein 
\cite[p.~333--336]{Spa1}. The ``very particular'' case is related to
the unipotent class $C$ specified as follows.
\[\begin{array}{c@{\hspace{9pt}}c@{\hspace{9pt}}
c@{\hspace{9pt}}c@{\hspace{9pt}}c@{\hspace{9pt}}c} \hline 
C & \dim C_G(u) & A(u) & |C_G(u)^F| & E_0 & E:\dim \cE_u\\ \hline
E_8(b_6)   & 28 &\fS_3& 6q^{28}, 2q^{28}, 3q^{28}  & 2240_{10} & 
\begin{array}{c} 175_{12}:2\\ 840_{13}:1 \end{array}\\\hline 
\end{array}\]
where, as before, the last two columns specify $E,E_0\in\Irr(W)$ such that 
$\iota_G(E_0)=(C,\overline{\Q}_\ell)$ and $\iota_G(E)=(C,\cE)$ with 
$\cE\not\cong \overline{\Q}_\ell$. (Note that Mizuno uses the notation 
$D_8(a_3)$ for this class; furthermore, the same conventions for the 
notation of $\Irr(W)$ apply as in Section~\ref{sec6}.) Up to conjugation 
by elements in $G^F$, there is a unique $u_0\in C^F$ such that 
$|C_G(u_0)^F|=6q^{28}$ and $F$ acts trivially on $A(u_0)$. By Mizuno 
\cite[Lemma~53]{Miz2}, such a representative is given by 
\begin{align*} 
u_0:=z_{77}& =x_{11110000}(1) x_{10111000}(1) x_{01111000}(1) 
x_{00111100}(1)\\ &\qquad\qquad\cdot 
x_{01011100}(1) x_{00011110}(1) x_{00000011}(1) x_{00000111}(1).
\end{align*}
(As before, if $\alpha=\sum_{i=1}^8 n_i\alpha_i\in \Phi$, we just write 
$x_{n_1n_2\ldots n_8}(t)$ instead of $x_\alpha(t)$.) We consider the
corresponding signs $\delta_E$ with respect to $u_0$; we claim that
\[ \delta_{2240_{10}}=1 \qquad\mbox{and}\qquad
\delta_{840_{13}}=(-1)^m \quad (q=2^m).\]
This is seen as follows. As before, since $u_0\in C^{F_2}$, it is 
sufficient to determine the signs in the special case where $m=1$ (see
Remark~\ref{basep}). As in Example~\ref{sub27}, we run the function 
{\tt ICCTable} which yields the coefficients $p_{E',E}$. By inspection of 
the output, and using the argument in Remak~\ref{critic}(a), we already 
see that $\delta_{2240_{10}}=1$.

We now follow the argument in \S \ref{sec64}. Let $u_0=u_1,u_2,u_3\in 
C^{F_2}$ be representatives of the $G^{F_2}$-conjugacy classes that are
contained in $C^{F_2}$, and let $a_1,a_2,a_3\in A(u_0)$ be corresponding 
representatives of the conjugacy classes of $A(u_0)$, where the notation 
is such that $a_2$ corresponds to a transposition in $\fS_3\cong A(u_0)$
and $a_3$ to a $3$-cycle. By Mizuno \cite[Lemma~53]{Miz2}, such 
representatives are given by 
\begin{align*}
u_2:=z_{78}&=z_{77}x_{00001111}(1),\\
u_3:=z_{79}&=x_{11110000}(1) x_{10111000}(1) x_{01111000}(1)x_{00111100}(1)
x_{01011100}(1)\\&\qquad\qquad\qquad \cdot x_{00011110}(1) x_{00000001}(1)
x_{00000111}(1)x_{00001111}(1).
\end{align*}
Using the output of {\tt ICCTable}, and setting $q=2$, we find the formula
\begin{align*}
Q_1(u_i) &=(2240q^{10}+3688q^9+3444q^8+2360q^7+1351q^6\\& \qquad\qquad 
+672q^5+ 294q^4+112q^3+35q^2+8q+1)Y_{2240_{10}}(u_i)\\&\qquad +
175q^{10}Y_{175_{12}}(u_i)+(840q^{10}+650q^9+160q^8)Y_{840_{13}}(u_i)\\
&= 5479485\,Y_{2240_{10}}(u_i)+179200\,Y_{175_{12}}(u_i)+1233920\,
Y_{840_{13}}(u_i) \qquad \mbox{($q=2$)}.
\end{align*}
for $1\leq i \leq 3$. Up to the signs $\delta_{175_{12}}$ and 
$\delta_{840_{13}}$, the values of the $Y$-functions on $u_i$ are given 
by character values of $\fS_3$; see Remark~\ref{lu11}. Thus, for $q=2$, 
we obtain
\begin{align*}
Q_1(u_1)&=5479485+2\cdot 179200\,\delta_{175_{12}}+1233920\, 
\delta_{840_{13}},\\
Q_1(u_2)&=5479485-1233920\, \delta_{840_{13}},\\
Q_1(u_3)&=5479485-179200\,\delta_{175_{12}}+1233920\, 
\delta_{840_{13}}.
\end{align*}
Assume, if possible, that $\delta_{840_{13}}=1$. Then we would have 
$Q_1(u_2)=5479485-1233920=4245565$ for $q=2$. On the other hand, running 
through all sets $Q_{1,w}(z_{78})$ (as in Lemma~\ref{rem31aa}) where 
$l(w)\leq 23$, we already find $4047101$ cosets that are fixed 
by~$z_{78}$. Running also through sets $Q_{1,w}(z_{78})$ where $l(w)=24$, 
we find further $305856$ cosets that are fixed. Thus, we have $Q_1(u_2)>
4245565$ and so we conclude that $\delta_{840_{13}}=-1$, as claimed. 

The total running time for these computations is about $1$~year, even
with efficiency improvements as in \S \ref{sec54}. (It already takes 
$4$ months just to deal with those sets $Q_{1,w}(z_{78})$
where $l(w)=24$.) Everything is much slower than in the previous cases 
because, for type $E_8$, we only have at our disposal the $248$-dimensional 
adjoint representation, whereas for type $E_7$ we could use the 
$56$-dimensional minuscule weight representation. However, distributing 
the task over a small number of (independent) standard desktop computers
 with altogether $30$~processors, we could manage to complete the 
computations in $2$~weeks.~---~Once $\delta_{840_{13}}$ is determined, we 
obtain $Q_1(u_1)=5479495-1233920+2\cdot 179200 \delta_{175_{12}}\in 
\{ 4603965,3887165\}$ (for $q=2$). If we could find strictly more than 
$3887165$ cosets that are fixed by $z_{77}$, then we would be able to 
conclude that $\delta_{175_{12}}=1$. However, this appears to be even 
more difficult computationally, the reason being that the difference 
between the lower and the upper bound for the value of $Q_1$ is much smaller 
than in the previous case. (So we would need to find \textit{almost all} 
cosets that are fixed by $z_{77}$.)

All this clearly indicates that type $E_8$ is not completely out of reach, 
but some more sophisticated algorithms are certainly required in order to 
deal with the remaining open cases (especially for $p=3,5$). An independent 
verification of the above results would also be highly desirable.


%
%
%



\begin{thebibliography}{131}

\bibitem{BeSp}
W. M. Beynon and N. Spaltenstein, \textit{Green functions of finite Chevalley
groups of type $E_n$ ($n=6,7,8$)}, J. Algebra {\bf 88} (1984), 584--614.

\bibitem{C1}
R. W. Carter, \textit{Simple groups of Lie type}, Wiley, New York, 1972;
reprinted 1989 as {W}iley {C}lassics {L}ibrary {E}dition.

\bibitem{C2} 
R. W.~Carter, \textit{Finite groups of Lie type: Conjugacy classes 
and complex characters}, Wiley, New York, 1985.

\bibitem{CMT}
A. M. Cohen, S. H. Murray and D. E. Taylor, \textit{Computing in groups of
Lie type}, Math. Comp. {\bf 73} (2004), 1477--1498.

\bibitem{DeLu}
P.~Deligne and G.~Lusztig, \textit{Representations of reductive groups over
finite fields}, Annals Math. {\bf 103} (1976), 103--161.

\bibitem{gap}
The GAP~Group, \textit{GAP -- Groups, Algorithms, and Programming, 
Version 4.10.0}, 2018, \verb+(https://www.gap-system.org)+.

\bibitem{aver}
M. Geck, \textit{On the average values of the irreducible characters of finite
groups of Lie type on geometric unipotent classes}, Doc.\ Math.\ J.\ DMV
{\bf 1} (1996), 293--317 (electronic).

\bibitem{mylie}
M. Geck, \textit{On the construction of semisimple Lie algebras and 
Chevalley groups}, Proc. Amer. Math. Soc.  {\bf 145} (2017), 3233--3247.

\bibitem{mylie1}
M. Geck, \textit{Minuscule weights and Chevalley groups}, {\it in:} Finite 
Simple Groups: Thirty Years of the Atlas and Beyond (Celebrating the Atlases 
and Honoring John Conway, November 2-5, 2015 at Princeton University),
pp.~159--176, Contemporary Math., vol. 694, Amer. Math. Soc., 2017.

\bibitem{chevlie}
M. Geck, \textit{{\sf ChevLie}~---~Constructing Lie algebras and Chevalley 
groups in GAP}, July 2016, available at 
\url{http://www.math.rwth-aachen.de/~CHEVIE/contrib.html}.

\bibitem{pbad}
M. Geck, \textit{On the values of unipotent characters in bad characteristic},
Rend. Cont. Sem. Mat. Univ. Padova (2019), online first, 
\url{DOI: 10.4171/RSMUP/14}.

\bibitem{hartur}
M. Geck, \textit{Green functions and Glauberman degree-divisibility};
preprint (April, 2019); see {\tt arXiv:1904.04586}.

\bibitem{chevie}
M.~Geck, G.~Hiss, F.~L\"ubeck, G.~Malle, and G.~Pfeiffer, 
{\sf CHEVIE}---\textit{A system for computing and processing generic 
character tables}, Appl. Algebra Engrg. Comm. Comput. {\bf 7} (1996), 
175--210; electronically available at
\url{http://www.math.rwth-aachen.de/~CHEVIE}

\bibitem{Law}
R. Lawther, \textit{Jordan block sizes of unipotent elements in exceptional
algebraic groups}, Comm. Algebra {\bf 23} (1995), 4125--4156.

\bibitem{LiSe}
M. W. Liebeck and G. M. Seitz, \textit{Unipotent and nilpotent classes in 
simple algebraic groups and Lie algebras}, Math. Surveys and Monographs,
vol. 180, Amer. Math. Soc., Providence, RI, 2012.

\bibitem{cbms}
G. Lusztig, \textit{Representations of finite Chevalley groups}, C.B.M.S.\
  Regional Conference Series in Mathematics, vol.~39, Amer. Math. Soc.,
  Providence, RI, 1977.

\bibitem{L1} 
G.~Lusztig, \textit{Characters of reductive groups over a finite field},
Ann.\ Math.\ Studies {\bf 107}, Princeton U.\ Press, 1984.

\bibitem{LuIC}
G. Lusztig, \textit{Intersection cohomology complexes on a reductive group},
Invent. Math. {\bf 75} (1984), 205--272.

\bibitem{L2d} 
G.~Lusztig, \textit{Character sheaves IV}, Adv.\ Math. {\bf 59} (1986), 1--63.

\bibitem{L2e} 
G.~Lusztig, \textit{Character sheaves V}, Adv.\ Math. {\bf 61} (1986), 
103--155.

\bibitem{Lintr}
G. Lusztig, \textit{Introduction to character sheaves}, in: \textit{The 
Arcata Conference on Representations of Finite Groups} (Arcata, Calif., 
1986), Proc. Sympos. Pure Math., 47, Part 1, Amer. Math. Soc., Providence, 
RI, 1987, pp.~164--179.

\bibitem{L5}
G. Lusztig, \textit{Green functions and character sheaves}, Ann. Math.
  {\bf131} (1990), 355--408.

\bibitem{Ldisc4}
G.~Lusztig, \textit{Character sheaves on disconnected groups, IV},
Represent. Theory {\bf 8} (2004), 145--178.

\bibitem{L10} 
G.~Lusztig, \textit{On the cleanness of cuspidal character sheaves}, 
Mosc. Math. J. {\bf 12} (2012), 621--631. 

\bibitem{L19}
G. Lusztig, \textit{The canonical basis of the quantum adjoint 
representation}, J. Comb. Alg. {\bf 1} (2017), 45--57.

\bibitem{LuSp1}
G. Lusztig and N. Spaltenstein, \textit{On the generalized Springer 
correspondence for classical groups}, in: \textit{Algebraic groups and 
related topics}, Adv. Stud. Pure Math. 6, North Holland and Kinokuniya 1985, 
pp.~289--316. 

\bibitem{Ma90}
G. Malle, \textit{Die unipotenten Charaktere von ${^2\!F}_4(q^2)$}. Comm.
Algebra {\bf 18} (1990), 2361--2381.

\bibitem{Mal1}
G. Malle, \textit{Green functions for groups of type $F_4$ and $E_6$ in
characteristic~$2$}, Comm. Algebra {\bf 21} (1993), 747--798.

\bibitem{MaSh}
R. M.~Marcelo and K.~Shinoda, \textit{Values of the unipotent characters of 
the Chevalley group of type $F_4$ at unipotent elements}, Tokyo J. Math.
{\bf 18} (1995), 303--340.

\bibitem{jmich}
J. Michel, \textit{The development version of the CHEVIE package of GAP3},
J. Algebra  {\bf 435} (2015), 308--336.  Webpage at 
\url{https://webusers.imj-prg.fr/~jean.michel/chevie/chevie.html}.

\bibitem{Miz}
K. Mizuno, \textit{The conjugate classes of Chevalley groups of type $E_6$},
J. Fac. Sci. Univ. Tokyo Sect. IA Math. {\bf 24} (1977), 525--563. 

\bibitem{Miz2}
K. Mizuno, \textit{The conjugate classes of unipotent elements of the
Chevalley groups $E_7$ and $E_8$}, Tokyo J. Math. {\bf 3} (1980), 391--461.

\bibitem{Por}
U. Porsch, \textit{Die Greenfunktionen der endlichen Gruppen $E_6(q)$, 
$q=3^n$}, Diplomarbeit, Universit\"at Heidelberg, 1993.

\bibitem{Sh74}
T. Shoji, \textit{The conjugacy classes of Chevalley groups of type $(F_4)$
over finite fields of characteristic $p\neq 2$}. J. Fac. Sci. Univ. Tokyo
Sect. IA Math. {\bf 21} (1974), 1--17.

\bibitem{Sh82}
T. Shoji, \textit{On the Green polynomials of a Chevalley group of type $F_4$},
Comm. Algebra {\bf 10} (1982), 505--543.

\bibitem{S1}
T. Shoji, \textit{Green functions of reductive groups over a finite field},
in: \textit{The Arcata Conference on Representations of Finite Groups} 
(Arcata, Calif., 1986),  Proc. Sympos. Pure Math., 47, Part 1, Amer. 
Math. Soc., Providence, RI, 1987, pp.~289--302.

\bibitem{S2} 
T.~Shoji, \textit{Character sheaves and almost characters of reductive 
groups}, Adv.\ Math. {\bf 111} (1995),  244--313.

\bibitem{S3} 
T.~Shoji,  \textit{Character sheaves and almost characters of reductive
groups, II},  Adv.\ Math. {\bf 111} (1995), 314--354.

\bibitem{S6a} 
T. Shoji, \textit{Generalized Green functions and unipotent classes for 
finite reductive groups, I}, Nagoya Math. J. {\bf 184} (2006), 155--198.

\bibitem{S6} 
T. Shoji, \textit{Generalized Green functions and unipotent classes for 
finite reductive groups, II}, Nagoya Math. J. {\bf 188} (2007), 133--170.

\bibitem{Spa1}
N. Spaltenstein, \textit{On the generalized Springer correspondence for 
exceptional groups}, in: \textit{Algebraic groups and related topics}, Adv. 
Stud. Pure Math. 6, North Holland and Kinokuniya, 1985, pp.~317--338. 

\bibitem{St}
R.~Steinberg, \textit{Lectures on Chevalley groups}. Mimeographed notes,
Department of Math., Yale University, 1967. Now available as vol.~66 of
the University Lecture Series, Amer. Math. Soc., Providence, R.I., 2016.

\end{thebibliography}
\end{document}